\documentclass[12pt,reqno]{amsart}
\usepackage{exscale}
\usepackage{relsize}
\usepackage{amsfonts}
\usepackage{a4wide}
\usepackage{mathrsfs,amsmath,makecell}
\allowdisplaybreaks
\numberwithin{equation}{section}
\newtheorem{theorem}{Theorem}[section]

\newtheorem{lemma}[theorem]{Lemma}

\theoremstyle{definition}

\newtheorem{remark}[theorem]{Remark}

\newcommand{\R}{\mathbb{R}}

\begin{document}

\title
 [Ground states for 3D DIPOLAR BOSE-EINSTEIN CONDENSATES]
 {Ground states for 3D DIPOLAR BOSE-EINSTEIN CONDENSATES INVOLVING QUANTUM FLUCTUATIONS AND THREE-BODY
losses} {\let\thefootnote\relax \footnotetext{This work was supported by National Natural Science Foundation of China (Grant No. 11901147 and 11771166) and the Fundamental Research Funds for the Central Universities of China (Grant No. JZ2020HGTB0030 and JZ2019HGBZ0156).}}

\maketitle
\begin{center}
\author{Xiao Luo}
\footnote{Email addresses: luoxiao@hfut.edu.cn (X. Luo).}
\author{Tao Yang}
\footnote{Email addresses: yangtao\_pde@163.com (T. Yang).}
\end{center}

\begin{center}
\address {1 School of Mathematics, Hefei University of Technology, Hefei, 230009, P. R. China}

\address {2 School of Mathematics and Statistics, Central China Normal University, Wuhan, 430079, P. R. China}
\end{center}
\maketitle

\begin{abstract}
We consider ground states of three-dimensional dipolar Bose-Einstein condensate involving quantum fluctuations and three-body losses, which can be described equivalently by positive $L^2$-constraint critical point of the Gross-Pitaevskii energy functional
\[E(u)\!=\!\frac{1}{2}\int_{{\mathbb{R}^3}} {|\nabla u|}^2dx+\frac{\lambda_{1}}{2}\int_{{\mathbb{R}^3}} {| u|}^4dx+\frac{\lambda_{2}}{2} \int_{\mathbb{R}^{3}}\left(K \star|u|^{2}\right)|u|^{2} d x+\frac{2\lambda_{3}}{p}\int_{{\mathbb{R}^3}} {|u|}^{p}dx,\]
where $2\!<\!p\!<\!\frac{10}{3}$, $\lambda_3\!\in\!\R^{-}$, $\star$ is the convolution, $ K(x) \!=\! \frac{{1-3{{\cos }^2}\theta(x) }}{{{{| x |}^3}}}$, $\theta(x)$ is the angle between the dipole axis determined by $(0,0,1)$ and the vector $x$.
If ${\lambda _1} \!\!<\!\! \frac{4\pi}
{3} {\lambda _2}\!\leq\! 0$ or ${\lambda _1} \!\!<\!- \frac{8\pi}{3} {\lambda _2}\!\leq\! 0$, $E(u)$
is unbounded on the $L^2$-sphere $S_{c}\!:=\!\Big\{ u \!\in\! H^1({\mathbb{R}^3}): \int_{{\mathbb{R}^3}} {{|u|}^2}dx\!=\!c^2  \Big\}$, so we turn to study a local minimization problem
$$ m(c,R_0)\!:=\!\inf _{u \in V^c_{R_0}} E(u)$$
for a suitable $R_0\!>\!0$ with $V^c_{R_0} \!:=\!\left\{u \!\in\! S_c : \big(\int_{{\mathbb{R}^3}} {{|\nabla u|}^2dx}\big)^{\frac{1}{2}} \!<\!R_0\right\}$. 

We show that $m(c,R_0)$ is achieved by some $u_c>0$, which is a stable ground state. Furthermore, by refining the upper bound of $m(c, R_0)$, we provide a precise description of the asymptotic behavior of $u_c$ as the mass $c$ vanishes, i.e.
 $$\Big[  \frac{p|\lambda_{3}|}{2\gamma_{c}}
\Big]^{\frac{1}{p-2}}u_{{c}}(\frac{x+y_c}{\sqrt{2\delta_p\gamma_{c}}})\!
\rightarrow\! W_p~~~~\mbox{in}~~~~H^1(\R^3)~~~~\mbox{for some}~~~~y_{c} \!\in\! \mathbb{R}^{3}~~~~\mbox{as}~~~~c\to 0^+,$$
where $W_p$ is the unique positive radial solution of $-\Delta W+(\frac{1}{\delta_p}-1)W \!=\!\frac{2}{p\delta_p}|W|^{p-2}W$ with $\delta_p\!=\!\frac{3(p-2)}{2p}$, $\mathcal{C}_{p}\!=\!\Big( \frac{p}{2 ||W_p||^{p-2}_{2}} \Big)^{\frac{1}{p}}$ and $\gamma_{c}\!=\!\Big[2\delta_p \Big]^{\frac{p\delta_p}{2-p\delta_p}}
\Big[\mathcal{C}_{p}^{p}|\lambda_{3}|\Big]^{\frac{2}{2-p\delta_p}}
c^{\frac{2(p-2)}{2-p\delta_p}}$.

{\bf Key words }: Dipolar Bose-Einstein condensate; Ground states; Stability; Asymptotic behavior.

{\bf 2010 Mathematics Subject Classification }: Primary 35J20, 35J60, 35B38, 35B40.
\end{abstract}

\maketitle

\section{Introduction and main result}

\setcounter{equation}{0}

This paper concerns the existence of solutions $({\mu},u)\!\in\!  {\mathbb{R}}\!\times \!H^1({\mathbb{R}^3})$ to the Gross-Pitaevskii equation involving quantum fluctuations and three-body interactions
\begin{equation}\label{eq1.1}
- \frac{1}{2} \Delta u +{\lambda _1}{| u |^2}u + {\lambda _2}(K \star {| u |^2})u+{\lambda _3}{| u |^{p-2}}u+\mu u=0~~~~{\text{ in }}{\mathbb{R}^3}
\end{equation}
under the constraint
\begin{equation}\label{eq1.2}
  \int_{{\mathbb{R}^3}} {{u}^2}=c^2,
\end{equation}
where $c\!>\!0$, $2\!<\!p\!<\!\frac{10}{3}$, $(\lambda_1, \lambda_2, \lambda_3)\!\in\! \R^2\!\times\!\R^{-}$,
$\star$ denotes the convolution, $ K(x) \!=\! \frac{{1-3{{\cos }^2}\theta(x) }}{{{{| x |}^3}}}$ and $\theta(x)$ is the angle between the dipole axis determined by $(0,0,1)$ and the vector $x$. We call $u$ a normalized solution to (\ref{eq1.1}), since (\ref{eq1.2}) imposes a normalization on its $L^2$-mass. Normalized solutions to (\ref{eq1.1}) can be obtained by searching critical points of
\begin{equation} \label{energyF}
E(u):=\frac{1}{2}\|\nabla u\|_{2}^{2}+\frac{\lambda_{1}}{2}\|u\|_{4}^{4}+\frac{\lambda_{2}}{2} \int_{\mathbb{R}^{3}}\left(K \star|u|^{2}\right)|u|^{2} d x+\frac{2\lambda_{3}}{p}\|u\|_{p}^{p}
\end{equation}
on the constraint
\begin{equation} \label{eq1.06}
S_{c}: =\Big\{ u \in H^1({\mathbb{R}^3}): {||u||}_2^2=\int_{{\mathbb{R}^3}} {u}^2=c^2  \Big\}
\end{equation}
with $\mu$ appearing as Lagrange multipliers. This fact implies that $\mu$ cannot be determined a priori, but is part of the unknown.

Problem (\ref{eq1.1})-(\ref{eq1.2}) arises from seeking standing waves, i.e. $\psi(t,x)=e^{i \mu t} u(x)$, for the time-dependent Gross-Pitaevskii equation
\begin{align}\label{gpt1.4}
i \partial_{t} \psi=-\frac{1}{2} \Delta \psi +\lambda_{1}|\psi|^{2} \psi+\lambda_{2}\left(K \star|\psi|^{2}\right) \psi+\lambda_{3}|\psi|^{p-2} \psi, \quad (t,x)\in \mathbb{R}^{+}\times\mathbb{R}^{3},
\end{align}
which models the dipolar Bose-Einstein condensates. The parameters $\lambda_i  (i=1,2,3)$ describe the strength of the three nonlinearities in \eqref{gpt1.4}. For $p\!=\!5$ and $\lambda_{3}\!>\!0$, \eqref{gpt1.4} corresponds to the Lee-Huang-Yang correction (see \cite{mBlo}). When $p\!=\!6$, \eqref{gpt1.4} with $\lambda_{3}\!>\!0$ describes the short-range conservative three-body interactions (see \cite{lAkI}) and \eqref{gpt1.4} with $\lambda_{3}\!<\!0$ models three-body losses (see \cite{MeLf}), respectively. For detailed physical backgrounds on \eqref{gpt1.4} and further references, one can refer to \cite{bcw,cms,YhXl,Ymas,lYmAs,sszl,yy1,yy2} and the references therein.


R. Carles et al. \cite{cms} concerned with the existence and uniqueness of solution to
\begin{align}\label{1.7}
i{\partial _t}\psi  = -\frac{1}
{2}\Delta \psi+\frac{a^2}
{2}{|x{|^2}}\psi  + {\lambda _1}|\psi {|^2}\psi  + {\lambda _2}(K \star |\psi {|^2})\psi,{\text{  }}\psi (0,x) = {\psi _0}(x)\in {H^1}({\mathbb{R}^3})
\end{align}
In the so-called stable regime ${\lambda _1} \!\ge\! \frac{4\pi}
{3} {\lambda _2} \!\ge\! 0$, the authors showed that \eqref{1.7} has a unique global solution. In the unstable regime ${{\lambda _1} \!<\! \frac{4\pi}{3}{\lambda _2}}$, they observed the possibility of finite time blow-up. Later on, the stationary equation of \eqref{1.7}
\begin{align}\label{1.5}
- \frac{1}{2}\Delta u + \frac{{{a^2}}}{2}{| x |^{2}}u + {\lambda _1}{| u |^2}u + {\lambda _2}(K \star {| u |^2})u + \mu u = 0, \text{ } x \in {\R^3},
\end{align}
with $a \ge 0$ and $\mu \in \R$ has been studied by many authors, see e.g. \cite{as,bcw,bj1,ch}.
For $a\!=\!0$ and $\mu\!>\!0$, P. Antonelli and C. Sparber \cite{as} considered \eqref{1.5} by studying the minimization problem
\[
\inf_{0\not \equiv u\in H^1(\R^3)}J(u):=\inf_{0\not \equiv u\in H^1(\R^3)} \frac{\|\nabla u\|_{2}^{3} \| u\|_{2}} { -\lambda_{1}\|u\|_{4}^{4}-\lambda_{2}\int_{\mathbb{R}^{3}}\left(K \star|u|^{2}\right)|u|^{2} d x },
\]
and they showed that \eqref{1.5} has a positive solution provided that either  ${\lambda_1}\!<\!\frac{4\pi}{3}{\lambda_2}$ and ${\lambda_2}\!>0$ or ${\lambda _1}\!<\!-\frac{8\pi}{3}{\lambda_2}$ and ${\lambda_2}\!<0$.
Further properties (symmetry, regularity, decay) of the solutions to \eqref{1.5} have been obtained.

Alternatively, R. Carles and H. Hajaiej \cite{ch} studied \eqref{1.5} with a prescribed $L^2$-norm and $a=1$. They mainly focused on a global minimization problem on a $L^2$-sphere. If ${\lambda _1} \!\geq\! \frac{4}{3}\pi {\lambda _2} \!>\! 0$ or ${\lambda _1}\!\geq\!-\frac{8}{3}\pi {\lambda _2}\!>\!0$, they proved that \eqref{1.5} has a non-negative minimal solution with Lagrange multiplier $\mu$. Moreover, the obtained solution is unique and Steiner symmetric.

In \cite{bj1}, J. Bellazzini and L. Jeanjean also studied \eqref{1.5} with a prescribed $L^2$-norm and $a\geq0$. Since they assumed that ${\lambda_1}\!<\!\frac{4\pi}{3}{\lambda_2}$ and ${\lambda_2}\!>0$ or ${\lambda _1}\!<\!-\frac{8\pi}{3}{\lambda_2}$ and ${\lambda_2}\!<0$, the corresponding energy functional is unbounded on a $L^2$-sphere. As a result, the Palais-Smale sequence of the corresponding energy functional may not be bounded. To this end, J. Bellazzini et al. constructed a special Palais-Smale sequence which is very close to the Pohozaev manifold. If $a=0$, they proved that \eqref{1.5} has a mountain pass type solution by using the monotonicity of the mountain pass level. The authors then proved that there exists some $a_0>0$ such that \eqref{1.5} has a topological local minimal solution and a mountain pass type solution if $a\in (0,a_0]$. They also gave some stable scattering and asymptotic results.



The very recent works of Y. M. Luo and A. Stylianou \cite{Ymas,lYmAs} were devoted to the study of \eqref{gpt1.4} in two cases:  ${\lambda_3}\!<\!0$ and $p\!=\!5$; ${\lambda_3}\!>\!0$ and $p \in (4, 6]$, respectively. In \cite{Ymas}, the authors adopted a mountain pass argument on a $L^2$-sphere  and constructed a special Palais-Smale sequences in searching a positive ground state as saddle point. In \cite{lYmAs}, they proved several existence and nonexistence of standing waves to \eqref{gpt1.4} in different parameter regimes. In addition, they also proved the global well-posedness and small data scattering of solutions to \eqref{gpt1.4} if $p\!\in\!(4,6)$.

In \cite{TrIa}, A. Triay studied the sharp existence of minimizers in generalized Gross-Pitaevskii theory with the Lee-Huang-Yang correction, that is $\lambda_{3}\!>\!0$ and $p\!=\!5$ in \eqref{gpt1.4}. The author proved that
there is no minimizer for $E|_{S_c}$ if $c\in(0,c_0)$ and $E|_{S_c}$ possesses a global minimizer $u_c$ provided $c\geq c_0>0$ for some $c_0=c_0(\lambda_1,\lambda_2,\lambda_3)$. In addition, $u_c$ is $C^{\infty}$ and decays exponentially. The exact lower and upper bounds of $c_0$ are also obtained.




In this paper, we consider the existence and asymptotic properties of normalized solutions to (\ref{eq1.1}) with $(\lambda_1, \lambda_2, \lambda_3)\!\in\! \R^2\!\times\!\R^{-}$ and $2\!<\!p\!<\!\frac{10}{3}$. To the best of our knowledge, problem (\ref{eq1.1})-(\ref{eq1.2}) in this regime has not been studied before.


Following Definition 1.1 in \cite{bj1}, we say that $u$ is a \textbf{ground state} of (\ref{eq1.1}) on $S_c$ if it is a solution to (\ref{eq1.1}) having minimal energy among all solutions which belong to $S_c$, that is
$$
\left.E'\right|_{S_{c}}(u)=0 \quad \text { and } \quad E(u)=\inf \left\{E(w): \left.E'\right|_{S_{c}}(w)=0 \text { and }  w\in S_{c}\right\}.
$$
The set of the ground states will be denoted by $Z_{c}$. $Z_{c}$ is \textbf{stable} if for every $\varepsilon>0$ there exists $\delta>0$ such that, for any $\psi_{0} \in {H^1}({\mathbb{R}^3})$ with $\inf _{v \in Z_{c}} \left\|\psi_{0}-v\right\|_{{H^1}({\mathbb{R}^3})}<\delta$, we have
$$
\sup _{t \in [0,T_{\psi}^{max})} \inf _{v \in Z_{c}}\|\psi(t, \cdot)-v\|_{{H^1}({\mathbb{R}^3})}<\varepsilon, 
$$
where $\psi(t, \cdot)$ denotes the solution to (\ref{gpt1.4}) with initial datum $\psi_{0}$ and $T_{\psi}^{max}$ denotes its maximal time for existence.

To state the main results, we denote
\begin{equation}   \label{Lambda}
B(u):=\int_{{\R^3}} {{\lambda_1}{{| u |}^4}}  +{\lambda_2} {(K * {{| u |}^2}){{| u |}^2}},~~~~~~~~\Lambda\!:=\!\frac{\max \big\{|\lambda_{1}-\frac{4 \pi}{3}\lambda_{2}|,|\lambda_{1}+\frac{8 \pi}{3}\lambda_{2}|\big\}}{(2 \pi)^{3}}
\end{equation}
and
\begin{equation}   \label{D_0}
D_0\!: =\! \Bigl\{ {({\lambda _1},{\lambda _2}) \!\in\! {\mathbb{R}^2}: {\lambda _1} \!<\! \frac{4\pi}
{3} {\lambda _2} \!\leq\! 0,~~~~~~~~{\text{ or }}~~~~~~~~ {\lambda _1} \!<\!- \frac{8\pi}
{3} {\lambda _2}\!\leq\! 0} \Bigr\}.
\end{equation}
For $(\lambda_1,\lambda_2,\lambda_3)\!\in\! D_0\!\times\!\R^{-}$ and  $2\!<\!p\!<\!\frac{10}{3}$, we introduce four positive constants:
\begin{align} \label{cost2.11}
&\delta_p\!:=\!\frac{3(p-2)}{2p},\ \ \ \ \ \ \ \ \ \ \ \ \ \ \ \ \ \ \ \ \ \ c_*\!:=\!\Big\{\frac{p}{4(3\!-\!p\delta_p)|\lambda_{3}|\mathcal{C}_{p}^p } \Big[\frac{(2\!-\!p\delta_p)}{(3\!-\!p\delta_p)\Lambda \mathcal{C}_{4}^4}\Big]^{2\!-\!p\delta_p} \Big\}^{\frac{1}{2(4-p)}} , \nonumber \\
&{t}_{c_*}:=\frac{2-p\delta_p}{(3-p\delta_p)\Lambda \mathcal{C}_{4}^4 c_*},\ \ \ \ \ \ \ \ \ \ \kappa_{p,\lambda_{3}}\!:=\!\frac{10-3p}{6(p-2)}
\Big[ 2\delta_p \mathcal{C}_{p}^p|\lambda_{3}| \Big]
^{\frac{2}{2-p\delta_p}},
\end{align}
where $\mathcal{C}_{p}\!>\!0$ is the best constant in the Gagliardo-Nirenberg inequality \eqref{equ2.2}. 


Our main results are as follows.
\begin{theorem}\label{th1.1}
Let $(\lambda_1,\lambda_2,\lambda_3)\!\in\! D_0\!\times\!\R^{-},$   $2\!<\!p\!<\!\frac{10}{3}$ and $0\!<\!c\!\leq\!c_*$. Then, $E|_{S_{c}}$ has a critical point $u_{c}$ at level $m(c, R_0)\!<\!-\kappa_{p,\lambda_{3}} c^{\frac{6-p}{2-p\delta_p}}$, which is an local minimizer of $E$ on 
$$
V^c_{R_0}:=\left\{u \in S_{c}: {||\nabla u||}_2<R_0\right\}
$$
for some $R_0=R_0(c)\in(0,\frac{c}{c_*} {t}_{c_*}]$. Moreover, we have  \\
$\textbf{(1)}$ ${u}_c$ is a ground state of (\ref{eq1.1}) on $S_{c}$, and any ground state is a local minimizer of $E|_{V^c_{R_0}}$; \\
$\textbf{(2)}$ ${u}_{c}>0$ and ${u}_{c}$ solves (\ref{eq1.1}) for some $\mu_c\in \Big(\kappa_{p,\lambda_{3}}c^{\frac{2(p-2)}{2-p\delta_p}}, \frac{1-\delta_p}{2\delta_p}
\Big[\frac{4(3-p\delta_p)|\lambda_{3}|\mathcal{C}_{p}^p }{p}\Big]^{\frac{2}{2-p\delta_p}} c^{\frac{2(p-2)}{2-p\delta_p}}\Big)$;\\
$\textbf{(3)}$ $m(c,R_0)\!\to\! 0^-$, $||\nabla u_{c}||_{2} \rightarrow 0^+$ as $\lambda_3 \rightarrow 0^{-}$.\\
$\textbf{(4)}$ the set of ground states is stable under the flow associated with problem \eqref{gpt1.4}.
\end{theorem}


Furthermore, we give a precise description about the minimizers of $m(c,R_0)$ as $c\!\to\! 0^+$.
\begin{theorem}\label{th1.3}
Let $(\lambda_1,\lambda_2,\lambda_3)\!\in\! D_0\!\times\!\R^{-}$,   $2\!<\!p\!<\!\frac{10}{3}$, $c_k \!\to\! 0^+$ as $k\!\to\!+\infty$ and $u_{c_k}\!\in\! V^{c_k}_{R_0}\!=\!\left\{u \!\in\! S_{c_k}: {||\nabla u||}_2\!<\!R_0(c_k) \right\}$ be a positive minimizer of $m(c_k,R_0)$ for each $k\!\in\! \mathbb{N}$, then  \\
\begin{align*}
\frac{m(c_k,R_0)}{c_k^{\frac{6-p}{2-p\delta_p}}}  \to  -\kappa_{p,\lambda_{3}},~~~~~~~~~~~~~~~~~~~~~~~
\frac{\mu_{c_k}}{c_k^{\frac{2(p-2)}{2-p\delta_p}}}\to \frac{p(1-\delta_p)}{2-p\delta_p}\kappa_{p,\lambda_{3}},
~~~~~~~~~~~~~~~~~~~~~~\frac{ |B(u_{c_k})| }{c_k^{\frac{6-p}{2-p\delta_p}}} \to0, \\
\frac{\left\|\nabla u_{c_k} \right\|_{2}^2}{c_k^{\frac{6-p}{2-p\delta_p}}} \to
\frac{2p\delta_p}{2-p\delta_p} \kappa_{p,\lambda_{3}},~~~~~~~~~~~~~
\frac{\left\| u_{c_k} \right\|_{p}^p}{c_k^{\frac{6-p}{2-p\delta_p}}}\to \frac{p}{(2-p\delta_p)|\lambda_{3}|}\kappa_{p,\lambda_{3}},
~~~~~~~~\mbox{as}~~~~~~k\!\to\!+\infty.
\end{align*}
Moreover, there exists a sequence of $\{y_{k}\} \!\subset\! \mathbb{R}^{3}$ such that $$v_k(x)\!:=\!\Big[  \frac{p|\lambda_{3}|}{2\gamma_{c_k}}
\Big]^{\frac{1}{p-2}}u_{{c_k}}(\frac{x+y_k}{\sqrt{2\delta_p\gamma_{c_k}}})\!
\rightarrow\! W_p~~~~\mbox{in}~~~~H^1(\R^3) \mbox{as}~~~~~~k\!\to\!+\infty,$$
where $W_p$ is the unique positive radial solution of $-\Delta W+(\frac{1}{\delta_p}-1)W \!=\!\frac{2}{p\delta_p}|W|^{p-2}W$ and $\gamma_{c_k}\!:=\!\Big[2\delta_p \Big]^{\frac{p\delta_p}{2-p\delta_p}}
\Big[\mathcal{C}_{p}^{p}|\lambda_{3}|\Big]^{\frac{2}{2-p\delta_p}}
c_k^{\frac{2(p-2)}{2-p\delta_p}}$.
 \end{theorem}

\begin{remark}\label{re1.3} 
Our main results complete the results of \cite{Ymas,lYmAs}, which studied  \eqref{gpt1.4} with $p \in (4, 6]$. Due to the dipolar term, the energy functional $E(u)$ is not invariant under rotations and this lack of symmetry prevents us from working in the radial space $H_{rad}^1(\R^3)$. In \cite{Ymas}, the authors obtained a saddle point of $E|_{S_c}$ and they recovered the compactness of the corresponding Palais-Smale sequences by using the monotonicity of the mountain pass level. In \cite{lYmAs}, the authors focused on searching global minimizers of $E|_{S_c}$. Under the assumptions of Theorem \ref{th1.1}, $E|_{S_c}$ admits a convex-concave geometry. Therefore, it is possible to expect the existence of a local minimizer and a mountain pass critical point for $E|_{S_c}$. For these reasons, the methods adopted in \cite{NSoa} and \cite{Ymas,lYmAs} are not suitable for proving Theorem \ref{th1.1}. Some new ideas are developed to study local minimization problems without symmetry. From the physical point of view, Theorem \ref{th1.1} $(4)$ shows that the introduction of the term ${\lambda _3}{| u |^{p-2}}u$ leads to a stabilization of a system which was originally unstable. In \cite{bj1}, J. Bellazzini et al. also realized such stabilization by adding a small trapping potential $\frac{{{a^2}}}{2}{| x |^{2}}$ (see \eqref{1.5}), which is crucial in compactness analysis. Theorem \ref{th1.3} provides a precise description of the asymptotic behavior of $u_c$ as the mass $c$ vanishes. Such type of limit behavior of minimizers corresponding to global minimization problems arising in two-dimensional Bose-Einstein condensate are well studied, see e.g. \cite{GyWy,GLwJ} and the references therein. Concerning with local minimization problems, this type of results is new, as far as we know. It is worth pointing out that our method is well adapted to equations without symmetry.

\end{remark}

\begin{remark}
The mass threshold value $c_*$ can be taken arbitrary large by taking $|\lambda_3|>0$ small enough. This fact implies that we allow a wider range of $c$ provided $|\lambda_3|>0$ is small.
\end{remark}

\begin{remark}
If $(\lambda_1,\lambda_2,\lambda_3)\!\in\! D_0\!\times\!\R^{-}$ and  $2\!<\!p\!<\!\frac{10}{3}$, we have $\inf _{u \in S_{c}} E(u)\!=\!-\infty$ (see \eqref{engytoinfty} below). In view of this, we are in the $L^2$-supercritical setting. Due to the lack of symmetry, it is still open to obtain the second critical point of $E|_{S_c}$ under the assumptions of Theorem \ref{th1.1}. As in the proof of Theorem 1.1 in \cite{as}, we claim in Theorem 1.1 that  \\
\indent $\textbf{(i)}$ ${u}_{c}$ is radially symmetric in the $(x_1,x_2)$-plane and axially symmetric with respect to the $x_3$-axis;\\
\indent $\textbf{(ii)}$ ${u}_{c} \in W^{2,q}\left(\mathbb{R}^{3}\right)$ for all $q\geq2$ and there exist positive constants $\nu$ and $C$ such that
$$
e^{\nu|x|}(|u_c(x)|+|\nabla u_c(x)|) \leq C, \forall x \in \mathbb{R}^{3}.
$$
In fact, the application of the Steiner symmetrization (see \cite{AlGi,rSaU,as}) implies claim $\textbf{(i)}$. Claim $\textbf{(ii)}$ follows from the arguments given in Step 6 of the proof of Theorem 2.4 in \cite{RpoL}. These methods are also adopted in \cite{Ymas,lYmAs} to prove similar results. 

\end{remark}

The proof of Theorem \ref{th1.1} is motivated by \cite{NSoa,eHtR}. We  give the outline of our proof. Observing that $E(u) \geq h_c(\|\nabla u\|_{2})$, it is useful to study the local minimization problem
$$
m(c, R_0):=\inf _{u \in V^c_{R_0}} E(u),
$$
where $V^c_{R_0}:=\left\{u \in S_{c}: {||\nabla u||}_2<R_0\right\}$. Here $R_0=R_0(c)>0$ is a zero point of
$$h_c(t)=\frac{1}{2}t^{2}-\frac{\Lambda \mathcal{C}_{4}^4 c }{2} t^{3}-\frac{2|\lambda_{3}|\mathcal{C}_{p}^p c^{p(1-\delta_p)} }{p}  t^{p\delta_p}.$$ 
To derive the compactness of the minimizing sequences for $m(c, R_0)$, the further step is to prove the sub-additivity of $m(c, R_0)$. This step is difficult since $R_0(c)$ depends on $c$. The author of \cite{NSoa} overcame this difficulty by using the coupled rearrangement of radially symmetric functions and a careful analysis of $R_0(c)$, however, some parameters are required to be very small. Adopting some new ideas of \cite{eHtR}, we obtain the exact lower and upper bounds of $R_0(c)$ given by
$$0<\bar{t}_c< R_0(c) <\frac{c}{c_*} {t}_{c_*}< {t}_{c_*}< {t}_c< R_1(c),~~~~\mbox{if}~~~~0\!<\!c\!<\!c_*;$$
$$0<\bar{t}_{c_*}< R_0(c)={t}_{c_*}=R_1(c),~~~~\mbox{if}~~~~c\!=\!c_*,$$
(see Lemma \ref{lem3.1} below). From the fact that $t_{c_*}$ is independent of $c$ (see \eqref{cost2.11}), we turn to study another local minimization problem
$$
m(c, {t}_{c_*}):=\inf _{u \in V^c_{{t}_{c_*}}} E(u)<0\leq\inf_{u \in \partial V^c_{{t}_{c_*}} } E(u),
$$
where $V^c_{{t}_{c_*}}:=\left\{u \in S_{c}: {||\nabla u||}_2<{t}_{c_*}\right\}$ and $\partial V^c_{{t}_{c_*}}:=\left\{u \in S_{c}: {||\nabla u||}_2={t}_{c_*}\right\}$. It is sufficient to study $m(c,{t}_{c_*})$ since $m(c,{t}_{c_*})=m(c,k)$ for any $ k \in [R_{0}(c), R_{1}(c)]$. Actually, we have
\begin{equation*} 
||\nabla u||_{2}\!\in\! [R_0(c), R_{1}(c)] \Longrightarrow E(u) \!\geq\! h_c(||\nabla u||_{2})\!\geq\!0\!>\!m(c,{t}_{c_*});
\end{equation*}
\begin{equation} \label{Outline2}
E(u)\!<\!0~~~~\mbox{and}~~~~||\nabla u||_{2}\!\leq\!R_{1}(c) \Longrightarrow ||\nabla u||_{2}\!<\!R_0(c)\!\leq\!\frac{c}{c_*} {t}_{c_*}.
\end{equation}
Then a scaling argument indicates the subadditivity of $m(c, {t}_{c_*})$, which is vital in excluding the dichotomy of any minimizing sequences corresponding to $m(c, {t}_{c_*})$ (see Lemmas \ref{LemA3.5}-\ref{LeMa3.4.1} below). Hence, the compactness of any minimizing sequences for $m(c, {t}_{c_*})$ follows. Also, Lemma \ref{LeMa3.4.1} indicates that
$m(c,{t}_{c_*})$ is attained by some $u_c \!\in\! V^c_{ {t}_{c_*}}$. Combined with \eqref{Outline2}, we have $||\nabla u_c||_{2}\!<\!R_0(c)$ and
$$ \inf _{u\in V^c_{R_0(c)}} E(u)\!=\!m(c,R_0(c))\!=\!m(c,  {t}_{c_*})\!=\!E(u_c). $$

To prove Theorem \ref{th1.3}, we borrow some ideas from \cite{Nbal}, which studied a global minimization problem associated to a fourth order Schr\"{o}dinger equation. The main ingredient is the refined upper bound of $m(c, R_0)$ (see Section 4), i.e.
$$m(c, R_0)\!<\!-\kappa_{p,\lambda_{3}} c^{\frac{6-p}{2-p\delta_p}}.$$
We reach this goal by estimating $m(c, R_0)$ with some suitable testing functions. Therefore, it is necessary to keep the testing functions staying in the admissible set $V^c_{R_0(c)}$, where the lower bound $\bar{t}_c<R_0(c)$ plays an important role in this procedure. Note also that, the estimate of the Lagrange multiplier $\mu_c$ and the precise description of the asymptotic behavior of $u_c$ all depend heavily on $m(c, R_0)\!<\!-\kappa_{p,\lambda_{3}} c^{\frac{6-p}{2-p\delta_p}}$.\\

The paper is organized as follows, in Section 2, we give some preliminary results. In Section 3, we prove the compactness of the minimizing sequences for $m(c,{t}_{c_*})$. In Section 4, we derive the refined upper bound of $m(c, R_0)$. We prove Theorems \ref{th1.1}-\ref{th1.3} in Section 5. \\

\noindent \textbf{Notations:}~~~~For $1 \!\le\! p \!<\!\infty $ and $u\!\in\!{L^p}({\R^3})$, we denote ${\left\| u \right\|_p}\!:=\! {({\int_{{\R^3}} {\left| u \right|} ^p})^{\frac{1}{p}}}$. The Hilbert spaces $H^{1}(\mathbb{R}^{3})$ and $H_{rad}^{1}(\mathbb{R}^{3})$ are defined as
$$H^{1}(\mathbb{R}^{3}) :=\{u \in L^{2}(\mathbb{R}^{3}) :\nabla u \in L^{2}(\mathbb{R}^{3})\},~~~~~~~~H_{rad}^{1}(\mathbb{R}^{3}) \!:=\!\{u(x) \!\in\! H^{1}(\mathbb{R}^{3}): u(x)\!=\!u(|x|)\},$$
with the inner product $(u,v): = \int_{{\R^3}} {\nabla u\nabla v}  + \int_{{\R^3}} {uv}$  and norm ${\left\| u \right\|} := (\left\| {\nabla u} \right\|_2^2 + \left\| u \right\|_2^2)^{\frac{1}{2}}$.
$H^{-1}({\R^3})$ is the dual space of $H^1({\R^3})$. The space $D^{1,2}(\mathbb{R}^{3})$ is defined as
$$D^{1,2}(\mathbb{R}^{3}) :=\{u \in L^6(\mathbb{R}^{3}) :\nabla u \in L^{2}(\mathbb{R}^{3})\},$$
with the semi-norm $||u||_{D^{1,2}(\mathbb{R}^{3})}\!=\!\left\| {\nabla u} \right\|_2$. We use $``\rightarrow"$ and $``\rightharpoonup"$ to denote the strong and weak convergence in the related function spaces respectively. $C$ and $C_{i}$ will denote positive constants. $\langle\cdot,\cdot\rangle$ denote the dual pair for any Banach space and its dual space.
$o_{n}(1)$ and $O_{n}(1)$ mean that $|o_{n}(1)|\to 0$ and $|O_{n}(1)|\leq C$ as $n\to+\infty$, respectively. The Fourier transform of $u\in L^{2}(\R^N)$ is denoted by $\hat{u}$.


\section{Preliminaries}

\setcounter{equation}{0}
In this section, we give some preliminary results. Define the Fourier transform of $u$ by
$\widehat u(\xi ) := \int_{{\R^3}} {{e^{ - ix \cdot \xi }}u(x)dx}$,
then we have

\begin{lemma} \label{2.1.}(\cite{cms}, Lemma 2.3)
The Fourier transform of $K$ is given by
\[\widehat K(\xi ) = \frac{{4\pi }}
{3}(3{\cos ^2}\theta  - 1) = \frac{{4\pi }}
{3}\Bigl( {\frac{{3\xi_3^2}}
{{|\xi {|^2}}} - 1} \Bigr) = \frac{{4\pi }}
{3}\Bigl( {\frac{{2\xi_3^2 - \xi_1^2 - \xi_2^2}}
{{|\xi {|^2}}}} \Bigr) \in \Bigl[ { - \frac{4}
{3}\pi ,\frac{8}
{3}\pi } \Bigr],\]
where $\theta$ is the angle between $\xi$ and the vector $(0,0,1)$.
\end{lemma}

\begin{lemma} \label{2.2.} 
Let $u\!\in\! H^1(\R^3)$, $B(u)$ and $\Lambda$ are defined in \eqref{Lambda}.  Then
$$ B(u) \!=\! \frac{1}
{{{{(2\pi )}^3}}}\int_{{\mathbb{R}^3}} {[{\lambda_1} \!+\! {\lambda_2}\widehat K(\xi )]\big|\widehat{{{|u|}^2}}{\big|^2}d\xi },~~~~~~~~|B(u)| \!\le\! \Lambda\| u \|_4^4.$$
Moreover, if $u\!\in\! S_c$ and $({\lambda _1},{\lambda _2}) \!\in\! D_0$, then we have $B(u)\!<\!0$, where $D_0$ is given in \eqref{D_0}.
\end{lemma}


\begin{proof}
The former is obtained by Plancherel identity (see e.g. \cite{bcd}, Theorem 1.25) and the detailed proof can be found in \cite{bj1,Ymas,lYmAs}. The latter is a direct conclusion of Lemma~\ref{2.1.}.
\end{proof}

\begin{lemma} (Gagliardo-Nirenberg inequality, \cite{Wein}) \label{lem2.2}
Let $p\!\in\!(2,6)$ and $\delta_p\!=\!\frac{3(p-2)}{2p}$. Then there exists a constant $\mathcal{C}_{p}\!=\!\Big( \frac{p}{2 ||W_p||^{p-2}_{2}} \Big)^{\frac{1}{p}}\!>\!0$ such that
\begin{equation} \label{equ2.2}
 ||u||_{p} \leq \mathcal{C}_{p} \left\|\nabla u\right\|_{2}^{\delta_p} \left\|u\right\|_{2}^{(1-\delta_p)}, \qquad \forall u \in {H}^{1}(\mathbb{R}^{3}),
\end{equation}
where $W_p$ is the unique positive radial solution of
$ -\Delta W\!+\!(\frac{1}{\delta_p}-\!1)W \!=\!\frac{2}{p\delta_p}|W|^{p-2}W$.
\end{lemma}

By using Lemma~\ref{lem2.2}, the energy functional introduced by \eqref{energyF} can be reduced to
$$
E(u)=\frac{1}{2}\|\nabla u\|_{2}^{2}+\frac{1}{2} \frac{1}{(2 \pi)^{3}} \int_{\mathbb{R}^{3}}\left(\lambda_{1}+\lambda_{2} \widehat{K}(\xi)\right)\left|\widehat{|u|^{2}}(\xi)\right|^{2} d \xi+\frac{2\lambda_{3}}{p}\|u\|_{p}^{p}.
$$
Letting $u_{t}(x)\!=\!t^{\frac{3}{2}}u(tx)$ for $t\!>\!0$, we see that $\inf _{u \in S_{c}} E(u)\!=\!-\infty$ if $(\lambda_1, \lambda_2, \lambda_3)\!\in\! D_0\times\R^{-}$ and $2\!<\!p\!<\!\frac{10}{3}$. In fact, for $(\lambda_1, \lambda_2, \lambda_3)\!\in\! D_0\times\R^{-}$ and $u \in S_{c}$, we have $B(u)\!<\!0$, $p\delta_p\!<\!2$ and
\begin{equation} \label{engytoinfty}
E(u_t)=\frac{t^2}{2}\|\nabla u\|_{2}^{2}+\frac{t^3}{2}B(u)+\frac{2\lambda_{3}t^{p\delta_p}}{p}\|u\|_{p}^{p} \to -\infty~~~~\mbox{as}~~~~t \to +\infty.
\end{equation}
The global minimization method used in \cite{ch,lYmAs} breaks down. Therefore, we introduce a class of local minimization problems
\begin{equation} \label{equ2.14}
 m(c,k):=\inf _{u \in V^c_{k}} E(u), \forall k>0,
\end{equation}
where
\begin{equation} \label{equ2.13}
V^c_{k} :=\left\{u \in S_c :||\nabla u||_{2}<k\right\},~~~~~~~~\partial V^c_{k} :=\left\{u \in S_c :||\nabla u||_{2}=k\right\}.
\end{equation}
For $\lambda_{3}\in \mathbb{R}^{-}$, Lemmas \ref{2.2.}-\ref{lem2.2} indicate  that
\begin{align} \label{eq2.3}
E(u) \geq\frac{1}{2}\|\nabla u\|_{2}^{2}-\frac{\Lambda \mathcal{C}_{4}^4 c }{2} \left\|\nabla u\right\|_{2}^{3} -\frac{2|\lambda_{3}|\mathcal{C}_{p}^p c^{p(1-\delta_p)} }{p}  \left\|\nabla u\right\|_{2}^{p\delta_p} ,~~~~\forall u\in S_c.
\end{align}
From (\ref{eq2.3}), it is useful to consider the function $h_c : \mathbb{R}^{+} \rightarrow \mathbb{R}$:
$$h_c(t)=\frac{1}{2}t^{2}-\frac{\Lambda \mathcal{C}_{4}^4 c }{2} t^{3}-\frac{2|\lambda_{3}|\mathcal{C}_{p}^p c^{p(1-\delta_p)} }{p}  t^{p\delta_p}.$$
Since $p\delta_p\!<\!2$ for $2\!<\!p\!<\!\frac{10}{3}$, we have that $h_c(0^+)\!=\!0^{-}$ and $h_c(+\infty)\!=\!-\infty $. 


\begin{lemma}\label{lem3.1}
Let $c\!>\!0$, $\lambda_{3}\!<\!0$, $2\!<\!p\!<\!\frac{10}{3}$,  $t_c\!=\!\frac{2\!-\!p\delta_p}{(3\!-\!p\delta_p)\Lambda \mathcal{C}_{4}^4 c}$ and $\bar{t}_c:=\Big[\frac{4|\lambda_{3}|\mathcal{C}_{p}^p c^{p(1-\delta_p)} }{p}\Big]^{\frac{1}{2-p\delta_p}}$. Then\\
\indent (i) If $c\!<\!c_*$, $h_c(t)$ has a local minimum at negative level and a global maximum at positive level; moreover, there exist $R_0=R_0(c)$ and $R_1=R_1(c)$ such that
$$0<\bar{t}_c< R_0 <\frac{c}{c_*} {t}_{c_*}< {t}_{c_*}< {t}_c< R_1,~~~~~~~~h_c(R_0)=0=h_c(R_1),~~~~~~~~h_c(t)\!>\!0 \Leftrightarrow t\!\in\! (R_0,R_1);$$
\indent (ii) If $c\!=\!c_*$, $h_{c_*}(t)$ has a local minimum at negative level and a global maximum at level $0$; moreover, we have
$h_{c_*}( {t}_{c_*})=0,~~~~~~~~h_{c_*}(t)\!<\!0 \Leftrightarrow t\!\in\!(0, {t}_{c_*}) \cup ({t}_{c_*},+\infty)$.
\end{lemma}

\begin{proof}
(i) We first prove that $h_c$ has exactly two critical points. In fact,
$$
h_c'(t)=0 \Longleftrightarrow \psi_c(t)=2|\lambda_{3}|\delta_p\mathcal{C}_{p}^p c^{p(1-\delta_p)}, ~~~~\mbox{with}~~~~\psi_c(t)=t^{2-p\delta_p}-\frac{3\Lambda \mathcal{C}_{4}^4 c }{2} t^{3-p\delta_p}.
$$
Clearly, $\psi_c \nearrow$ on $[0,\hat{t}_c)$, $\searrow$ on $(\hat{t}_c,+\infty)$, where $\hat{t}_c\!=\!\frac{2(2\!-\!p\delta_p)}{3(3\!-\!p\delta_p)\Lambda \mathcal{C}_{4}^4 c}$. 
Since $p\delta_p\!<\!2$, we have $\max_{t\geq0}\psi_c(t)\!=\!\psi_c(\hat{t}_c)
\!=\!\frac{\hat{t}_c^{2\!-\!p\delta_p}}
{3\!-\!p\delta_p}\!>\!2|\lambda_{3}|\delta_p\mathcal{C}_{p}^p c^{p(1-\delta_p)}$ provided
\begin{align} \label{c^*}
c\!<\!c^*\!:=\!\Big\{ \frac{1}{2|\lambda_{3}|\delta_p(3\!-\!p\delta_p)\mathcal{C}_{p}^p } \Big[\frac{2(2\!-\!p\delta_p)}{3(3\!-\!p\delta_p)\Lambda \mathcal{C}_{4}^4}\Big]^{2\!-\!p\delta_p} \Big\}^{\frac{1}{2(4-p)}}.
\end{align} 
As $\psi_c(0^+)\!=\!0^+$ and $\psi_c(+\infty)\!=\!-\infty$, we see that $h_c$ has exactly two critical points if $c\!<\!c^*$.

Notice that
$$ h_c(t)>0 \Longleftrightarrow \varphi_c(t)>\frac{2|\lambda_{3}|\mathcal{C}_{p}^p c^{p(1-\delta_p)} }{p}, ~~~~\mbox{with}~~~~\varphi_c(t)=\frac{1}{2}t^{2-p\delta_p}-\frac{\Lambda \mathcal{C}_{4}^4 c }{2} t^{3-p\delta_p}.$$
Letting $ t_c\!=\!\frac{2\!-\!p\delta_p}{(3\!-\!p\delta_p)\Lambda \mathcal{C}_{4}^4 c}$ and $c_*\!:=\!\Big\{ \frac{p}{4(3\!-\!p\delta_p)|\lambda_{3}|\mathcal{C}_{p}^p } \Big[\frac{(2\!-\!p\delta_p)}{(3\!-\!p\delta_p)\Lambda \mathcal{C}_{4}^4}\Big]^{2\!-\!p\delta_p} \Big\}^{\frac{1}{2(4-p)}}$, by the fact that
$$
c\!<\!c_* \Longleftrightarrow \max_{t\geq0}\varphi_c(t)=\varphi_c( {t}_c)
\!=\!\frac{{t}_c^{2\!-\!p\delta_p}}
{2(3\!-\!p\delta_p)}>\frac{2|\lambda_{3}|\mathcal{C}_{p}^p c^{p(1-\delta_p)} }{p},
$$
we have $h_c(t)\!>\!0$ on an open interval $(R_0,R_1)$ if and only if $c\!<\!c_*$. Since $(\frac{3}{2})^{p\delta_p}>\frac{9}{8}p\delta_p$, we get $c_*\!<\!c^*$. Combined with $h_c(0^+)=0^{-}$ and $h_c(+\infty)=-\infty$, we see that $h_c$ has a local minimum point at negative level in $(0,R_0)$ and a global maximum point at positive level in $(R_0,R_1)$.  We also deduce that $\varphi_c( {t}_{c_*})>\varphi_{c_*}( {t}_{c_*})=\frac{2|\lambda_{3}|\mathcal{C}_{p}^p c_*^{p(1-\delta_p)} }{p}>\frac{2|\lambda_{3}|\mathcal{C}_{p}^p c^{p(1-\delta_p)} }{p}$ and
\begin{align*}
 \varphi_c\big(\frac{c}{c_*} {t}_{c_*}\big)
&=\frac{{t}_{c_*}^{2\!-\!p\delta_p}}
{2}\Big[1-\frac{ 2-p\delta_p}{3-p\delta_p}
\big(\frac{c}{c_*}\big)^2 \Big]\big(\frac{c}{c_*}\big)^{2-p\delta_p}
>\frac{{t}_{c_*}^{2\!-\!p\delta_p}}
{2(3\!-\!p\delta_p)} \big(\frac{c}{c_*}\big)^{2-p\delta_p} \\
&=\frac{2|\lambda_{3}|\mathcal{C}_{p}^p c_*^{p(1-\delta_p)} }{p}\big(\frac{c}{c_*}\big)^{2-p\delta_p}=\frac{2|\lambda_{3}|\mathcal{C}_{p}^p c^{p(1-\delta_p)} }{p}\big(\frac{c_*}{c}\big)^{p-2} \\
&>\frac{2|\lambda_{3}|\mathcal{C}_{p}^p c^{p(1-\delta_p)} }{p},
\end{align*}
which gives $h_c({t}_{c_*})\!>\!0$, $h_c(\frac{c}{c_*} {t}_{c_*})\!>\!0$, and hence $0\!<\! R_0 \!<\!\frac{c}{c_*} {t}_{c_*}\!<\! {t}_{c_*}\!<\! {t}_c\!< \! R_1$. Finally, the fact that   $h_c(t)\!\leq\!g_c(t)=\frac{1}{2}t^{2}-\frac{2|\lambda_{3}|\mathcal{C}_{p}^p c^{p(1-\delta_p)} }{p}t^{p\delta_p}$ leads to $R_0\!>\!\bar{t}_c$, where $\bar{t}_c:=\Big[\frac{4|\lambda_{3}|\mathcal{C}_{p}^p c^{p(1-\delta_p)} }{p}\Big]^{\frac{1}{2-p\delta_p}}$.


(ii) We omit the details as the proof is similar to that of (i). We emphasis that
$$
R_0\big(c_*\big)= {t}_{c_*}=R_1\big(c_*\big),~~~~~~~~\varphi_{c_*}( {t}_{c_*})=\frac{2|\lambda_{3}|\mathcal{C}_{p}^p c_*^{p(1-\delta_p)} }{p},~~~~~~~~\psi_{c_*}( {t}_{c_*})>2|\lambda_{3}|\delta_p\mathcal{C}_{p}^p c_*^{p(1-\delta_p)}.
$$
\end{proof}

Take $c\!=\!c_*$ in ${t}_{c}$ (see Lemma \ref{lem3.1}), then ${t}_{c_*}\!:=\!\!=\!\frac{2\!-\!p\delta_p}{(3\!-\!p\delta_p)\Lambda \mathcal{C}_{4}^4 c_*}$.


\begin{lemma}\label{LemA3.4}
Let $(\lambda_1, \lambda_2, \lambda_3)\!\in\! D_0\!\times\!\R^{-}$, $2\!<\!p\!<\!\frac{10}{3}$ and $0\!<\!c\!\leq\!c_*$. Then
$$ -\infty<m(c, {t}_{c_*}):=\inf _{u \in V^c_{{t}_{c_*}}} E(u)<0\leq\inf_{u \in \partial V^c_{{t}_{c_*}} } E(u).$$
\end{lemma}

\begin{proof}
If $u\!\in\! V^c_{{t}_{c_*}}$, we have $u_{s}(x)\!:=\!s^{\frac{3}{2}}u(sx)\!\in\! S_c$, $||\nabla u_{s} ||_{2}\!<\! {t}_{c_*}$ and $E(u_{s})\!<\!0$ for $s \!>\!0$ sufficiently small. So $m(c, {t}_{c_*})\!<\!0$. To prove $m(c, {t}_{c_*})\!>\!-\infty$, we deduce from (\ref{eq2.3}) that
$$
E(u) \geq h_c\left(||\nabla u||_{2}\right) \geq \min _{t \in\left[0, {t}_{c_*}\right]} h_c(t)>-\infty.
$$
If $u \!\in\! \partial V^c_{{t}_{c_*}}$, we get $u \!\in\! S_c$, $||\nabla u||_{2}\!=\!{t}_{c_*}$, and hence $E(u)\!\geq\! h_{c}( {t}_{c_*})\!\geq\! h_{c_*}({t}_{c_*})\!=\!0$. \end{proof}

\section{Compactness of local minimizing sequences}
In this Section, we prove the compactness of the minimizing sequences for $m(c, {t}_{c_*})$. 

\begin{lemma}\label{LemA3.5}
Let $(\lambda_1, \lambda_2, \lambda_3)\!\in\! D_0\!\times\!\R^{-}$, $2\!<\!p\!<\!\frac{10}{3}$ and $0\!<\!c\!\leq\!c_*$. Then \\
\indent (i) The map $c\mapsto m(c, {t}_{c_*})$ is continuous on $(0,c_*]$;\\
\indent (ii) If $c_1\!\in\!(0,c)$ and $c_2\!=\!\sqrt{c^2-c^2_1}$, we have $ m(c, {t}_{c_*})\leq m(c_1, {t}_{c_*})+m(c_2, {t}_{c_*})$. If $m(c_1, {t}_{c_*})$ or $m(c_2, {t}_{c_*})$ is attained, then $ m(c, {t}_{c_*})<m(c_1, {t}_{c_*})+m(c_2, {t}_{c_*})$.
\end{lemma}

\begin{proof} (i) For any $c\in(0,c_*]$ and $\{c_n\}\subset(0,c_*]$ such that $c_n \to c$, we aim at proving $m(c_n, {t}_{c_*}) \to m(c, {t}_{c_*})$. For any $\varepsilon > 0$ sufficiently small, we deduce from $m(c_n,{t}_{c_*})<0$ that, there exists a sequence $\{u_n\}\subset V^{c_n}_{{t}_{c_*}}:=\left\{w \in S_{c_n}: ||\nabla w||_{2}<{t}_{c_*}\right\}$ such that
\begin{equation} \label{CoTq1}
E(u_n)\leq m(c_n, {t}_{c_*})+\varepsilon
~~~~\mbox{and}~~~~E(u_n)<0.
\end{equation}
Denote $v_n:=\frac{c}{c_n}u_n$ and therefore $v_n\in S_c$. We claim that $||\nabla v_n||_{2}\!<\! {t}_{c_*}$. If $c_n\geq c$, then
$$||\nabla v_n||_{2}=\frac{c}{c_n}||\nabla u_n||_{2}\leq||\nabla u_n||_{2}\!<\! {t}_{c_*}.$$
If $c_n<c\leq c_*$, then Lemma \ref{lem3.1} (i) implies that $h_{c_n}(t)\!>\!0$ for $t\!\in\![\frac{{c_n}}{c_*} {t}_{c_*}, {t}_{c_*}]$. Since $ 0>E(u_n)\geq h_{c_n}({||\nabla u_n||}_2)$,
we deduce that $||\nabla u_n||_{2}\!<\!\frac{c_n}{c_*} {t}_{c_*}$.
Then, we get 
$$||\nabla v_n||_{2}=\frac{c}{c_n}||\nabla u_n||_{2}<\frac{c}{c_n} \frac{c_n}{c_*} {t}_{c_*}\leq{t}_{c_*}, $$
and hence $\{v_n\}\subset V^c_{{t}_{c_*}}$. The facts $c_n \to c$, $\{u_n\}\subset S_{c_n}$ and $||\nabla u_n||_{2}\!<\! {t}_{c_*}$ also lead to
\begin{align*}    
  ||u_n||_{p} \leq \mathcal{C}_{p} \left\|\nabla u_n\right\|_{2}^{\delta_p} \left\|u_n\right\|_{2}^{(1-\delta_p)} \leq C,~~~~~~~~|B(u_n)| \!\le\! \Lambda\|u_n\|_4^4\leq C
\end{align*}
for some $C>0$, which is independent of $n$. So, we get
\begin{equation} \label{CoTq2}
E(v_n)=\frac{1}{2}\frac{c^2}{c^2_n}{||\nabla u_n||}_2^2+\frac{1}{2} \frac{c^4}{c^4_n}B(u_n)-\frac{2|\lambda_{3}|}{p}\frac{c^p}{c^p_n}\|u_n\|_{p}^{p}
=E(u_n)+o_n(1).
\end{equation}
By using \eqref{CoTq1}-\eqref{CoTq2}, $\{v_n\}\subset V^{c }_{{t}_{c_*}}:=\left\{w \in S_{c}: ||\nabla w||_{2}<{t}_{c_*}\right\}$, we have
$$
 m(c, {t}_{c_*})\leq E(v_n)=E(u_n)+o_n(1) \leq m(c_n, {t}_{c_*})+\varepsilon+o_n(1).
$$
On the other hand, let $u \in V^{c}_{{t}_{c_*}}:=\left\{w \in S_{c}: ||\nabla w||_{2}<{t}_{c_*}\right\}$ be such that
\begin{equation} \label{CoTq3}
E(u)\leq m(c, {t}_{c_*})+\varepsilon
~~~~\mbox{and}~~~~E(u)<0.
\end{equation}
Set $w_n\!:=\!\frac{c_n}{c}u$ and then $w_n\!\in\! S_{c_n}$. We also have $||\nabla w_n||_{2}\!<\!{t}_{c_*}$ by using Lemma \ref{lem3.1}. In a fashion similar to \eqref{CoTq2}, we have $E(w_n)=E(u)+o_n(1)$. Then, \eqref{CoTq3} implies that
$$ m(c_n, {t}_{c_*})\leq E(w_n)=E(u)+o_n(1)\leq m(c, {t}_{c_*})+\varepsilon+o_n(1). $$
Letting $\varepsilon\to 0$ and then $n\to+\infty$, we have $ m(c_n, {t}_{c_*})\to m(c, {t}_{c_*})$.

(ii) For any fixed $c_1\in(0,c)$, we claim that
\begin{equation} \label{CoTq4}
 m(\theta c_1, {t}_{c_*})\leq \theta^2 m(c_1, {t}_{c_*}),~~~~\forall \theta \in (1,\frac{c}{c_1}]
\end{equation}
and that, if $m(c_1, {t}_{c_*})$ is attained, the inequality is strict. In fact, \eqref{CoTq4} leads to
\begin{align*}
 m(c, {t}_{c_*})\!=\!\frac{c^2_1}{c^2}m\Big(\frac{c}{c_1}c_1, {t}_{c_*}\Big) \!+\!\frac{c_2^2}{c^2}  m\Big(\frac{c}{c_2}c_2, {t}_{c_*}\Big)
\!\leq \!m(c_1, {t}_{c_*})\!+\!m(c_2, {t}_{c_*}),
\end{align*}
with a strict inequality if $m(c_1, {t}_{c_*})$ is attained. For any $\varepsilon > 0$ sufficiently small, there exists $u\in V^{c_1 }_{{t}_{c_*}}:=\left\{w \in S_{c_1 }: ||\nabla w||_{2}<{t}_{c_*}\right\}$ such that
\begin{equation} \label{CoTq5}
E(u)\leq m(c_1, {t}_{c_*})+\varepsilon
~~~~\mbox{and}~~~~E(u)<0.
\end{equation}
Similar to (i), Lemma \ref{lem3.1} (i) implies that $||\nabla u||_{2}\!<\!\frac{c_1}{c_*} {t}_{c_*}$. Set $v=\theta u$, then we have $$||v||_{2}=\theta||u||_{2}=\theta c_1,~~~~~~~~||\nabla v||_{2}\!=\!\theta||\nabla u||_{2}\!<\!\theta\frac{c_1}{c_*} {t}_{c_*}\!
\leq\! {t}_{c_*},$$
and therefore $v\in V^{\theta c_1 }_{{t}_{c_*}}:=\left\{w \in S_{\theta c_1 }: ||\nabla w||_{2}<{t}_{c_*}\right\}$. Finally, we deduce that
\begin{align*}
m(\theta c_1, {t}_{c_*})&\leq E(v)=\frac{\theta^2}{2} {||\nabla u||}_2^2+\frac{\theta^4}{2}B(u)-\frac{2|\lambda_{3}| \theta^p}{p}\|u\|_{p}^{p} \\
&<\frac{\theta^2}{2} {||\nabla u||}_2^2+\frac{\theta^2}{2}B(u)-\frac{2|\lambda_{3}| \theta^2}{p}\|u\|_{p}^{p}=\theta^2E(u) \leq \theta^2 m(c_1, {t}_{c_*})+\theta^2 \varepsilon.
\end{align*}
Since $\varepsilon>0$ is arbitrary, we have that $m(\theta c_1, {t}_{c_*})\leq\theta^2 m(c_1, {t}_{c_*})$. If $m(c_1, {t}_{c_*})$ is attained, then we can let $\varepsilon=0$ in \eqref{CoTq5} and thus the strict inequality follows.
\end{proof}

Applying Lemma \ref{LemA3.5}, we prove the compactness of the minimizing sequences for $m(c, {t}_{c_*})$.

\begin{lemma} \label{LeMa3.4.1}
Let $(\lambda_1, \lambda_2, \lambda_3)\!\in\! D_0\!\times\!\R^{-}$, $2\!<\!p\!<\!\frac{10}{3}$ and $0\!<\!c\!\leq\!c_*$. Then any sequence $\{u_{n}\} \subset H^1(\R^3)$ such that
$$\left\|u_{n}\right\|_{2} \to c,~~~~~~~~||\nabla u_n||_{2}< {t}_{c_*},~~~~~~~~E\left(u_{n}\right) \to m(c, {t}_{c_*})<0$$
is relatively compact in $H^1(\R^3)$ up to translations, and hence $m(c, {t}_{c_*})$ is attained.  
\end{lemma}

\begin{proof}
It is obvious that $\{u_{n}\}$ is bounded in $H^1(\R^3)$. If $\mathop {\lim }\limits_{n  \to \infty} \sup _{y \in \mathbb{R}^{3}} \int_{B_{R}(y)}\left|u_{n}(x)\right|^{2}dx\!=\!0$ for any $R\!>\!0$, we can prove that $\left\|u_{n}\right\|_{r} \to 0$ since $2<r<6$ (see Lemma I.1 in \cite{LilP}). Then we have $|B(u_n)| \!\le\! \Lambda\|u_n\|_4^4\to0$ and ${||u_n||}_p^p\to0$. But this leads to a contradiction that
\begin{align*}
0&>\mathop {\lim }\limits_{n  \to \infty}E_{\mu}(u_n)=\frac{1}{2}\mathop {\lim }\limits_{n  \to \infty}{||\nabla u_n||}_2^2+\frac{1}{2}\mathop {\lim }\limits_{n  \to \infty}B(u_n)-\frac{2|\lambda_{3}|}{p}\mathop {\lim }\limits_{n  \to \infty}{||u_n||}_p^p \geq0.
\end{align*}
Then, there exist an $\varepsilon_0>0$ and a sequence $\{y_{n}\} \subset \mathbb{R}^{3}$ such that
$$
\int_{B_{R}(y_n)}\left|u_{n}(x)\right|^{2} d x \geq \varepsilon_0>0
$$
for some $R>0$. Hence we have $u_{n}(x+y_n)\rightharpoonup u_c\not\equiv0$ in $H^1(\R^3)$ for some $u_c\in H^1(\R^3)$. Let $v_n(x):=u_{n}(x+y_n)-u_c$, then we have $v_{n}\rightharpoonup 0$ in $H^1(\R^3)$. Therefore, we get
\begin{align} \label{NlaG1} \left\|u_{n}\right\|_{2}^2=\left\|u_{n}(\cdot+y_n)\right\|_{2}^2=\left\|v_{n}\right\|_{2}^2
+\left\|u_c\right\|_{2}^2+o_n(1),  \nonumber \\
\left\|\nabla u_{n}\right\|_{2}^2=\left\|\nabla u_{n}(\cdot+y_n)\right\|_{2}^2=\left\|\nabla v_{n}\right\|_{2}^2
+\left\|\nabla u_c\right\|_{2}^2+o_n(1).
\end{align}
By the Br\'{e}zis-Lieb lemma \cite{a7}, we have
\begin{align} \label{NlaG2} \left\|u_{n}\right\|_{p}^p=\left\|u_{n}(\cdot+y_n)\right\|_{p}^p=\left\|v_{n}\right\|_{p}^p
+\left\|u_c\right\|_{p}^p+o_n(1).
\end{align}
From the proof of Proposition 3.2 in \cite{as}, it holds that
\begin{align} \label{Bs4}
B(u_n)=B(v_n)+B(u_c).
\end{align}
We now prove that
$$u_{n}(x+y_n)\rightarrow u_c\not\equiv0~~~~\mbox{in}~~~~L^2(\R^N),$$
or equivalently $v_{n}\!\rightarrow \!0$ in $L^2(\R^N)$. Denote $\left\|u_c\right\|_{2}\!=\!c_1$. If $c_1\!=\!c$, the proof is completed. If $c_1\!<\!c$, we learn from \eqref{NlaG1} that $\left\|v_{n}\right\|_{2}\!\leq\! c$ and $\left\|\nabla v_{n}\right\|_{2} \!\leq\! \left\|\nabla u_{n}\right\|_{2} \!<\! {t}_{c_*}$ for $n$ sufficiently large. Also, \eqref{NlaG1}-\eqref{Bs4} imply that
$$E\left(u_{n}\right)=E\big(u_{n}(\cdot+y_n)\big)
=E\left(v_{n} \right)+E\left(u_{c} \right)+o_n(1).$$
Since $E\left(u_{n}\right)\to m(c,{t}_{c_*})$ and $v_{n} \!\in\! V^{\left\| v_{n}\right\|_{2}}_{{t}_{c_*}}\!:=\!\left\{w \!\in\! S_{\left\| v_{n}\right\|_{2}}: ||\nabla w||_{2}\!<\!{t}_{c_*}\right\}$, we obtain
\begin{align} \label{NlaG3}
m(c,{t}_{c_*})=E\left(v_{n} \right)+E\left(u_{c} \right)+o_n(1)\geq m(\left\| v_{n}\right\|_{2}, {t}_{c_*})+E\left(u_{c} \right)+o_n(1).
\end{align}
By the continuity of $c\mapsto m(c,{t}_{c_*})$ (see Lemma \ref{LemA3.5} (i)), we can reduce \eqref{NlaG3} to
\begin{align} \label{NlaG4}
m(c,{t}_{c_*})\geq m(c_2,{t}_{c_*})+E\left(u_{c}\right),
\end{align}
where $c_2\!=\!\sqrt{c^2-c^2_1}\!>\!0$. The facts  $\left\|u_{c}\right\|_{2}\!=\!c_1$ and $\left\|\nabla u_{c}\right\|_{2} \!\leq\! \left\|\nabla u_{n}\right\|_{2} \!<\! {t}_{c_*}$ indicate that $E\left(u_{c} \right)\!\geq\! m(c_1, {t}_{c_*})$. If $E\left(u_{c} \right)\!>\! m(c_1, {t}_{c_*})$, \eqref{NlaG4} and Lemma \ref{LemA3.5} give a contradiction that
\begin{align*}
m(c, {t}_{c_*})\geq m(c_2, {t}_{c_*})+E\left(u_{c} \right)> m(c_2, {t}_{c_*})+m(c_1, {t}_{c_*})\geq m(c, {t}_{c_*}).
\end{align*}
If $E\left(u_{c} \right)=m(c_1, {t}_{c_*})$, that is, $m(c_1, {t}_{c_*})$ is attained by $u_{c}$, then Lemma \ref{LemA3.5} (ii) gives $m(c_2, {t}_{c_*})+m(c_1, {t}_{c_*})> m(c, {t}_{c_*})$ and hence we obtain a contradiction that
\begin{align*}
m(c, {t}_{c_*})\geq m(c_2, {t}_{c_*})+E\left(u_{c} \right)=m(c_2, {t}_{c_*})+m(c_1, {t}_{c_*})>m(c, {t}_{c_*}).
\end{align*}
Therefore, we have $\left\|u_c\right\|_{2}\!=\!c$ and hence $v_{n}\!\rightarrow \!0$ in $L^2(\R^N)$. It results that
$$||v_n||_{p} \leq \mathcal{C}_{p} \left\|\nabla v_n\right\|_{2}^{\delta_p} \left\|v_n\right\|_{2}^{(1-\delta_p)} \to0,~~~~~~~~|B(v_n)| \!\le\! \Lambda\|v_n\|_4^4\to0.$$
Finally, we get 
\begin{align*}
m(c,{t}_{c_*})=E\left(v_{n} \right)+E\left(u_{c} \right)+o_n(1)\geq\frac{1}{2}{||\nabla v_{n}||}_2^2+m(c,{t}_{c_*})+o_n(1),
\end{align*}
which indicates ${||\nabla v_{n}||}_2\leq o_n(1)$. So we have $v_{n}\!\rightarrow \!0$ in $H^1(\R^3)$.
\end{proof}

\section{Refined upper bound of the energy}
In this Section, we give a refined upper bound of $m(c,R_0)$. Let $c\!>\!0$, $\lambda _3\!<\!0$ and $2\!<\!p\!<\!\frac{10}{3}$ be fixed and search for $(\beta_c, v_c)\in \mathbb{R}\times H^1(\mathbb{R}^3)$, with $\beta_c>0$, solving
\begin{equation}\label{revisedI}
\left\{ \begin{gathered}
-\frac{1}{2} \Delta v+\beta_c v=|{\lambda _3}|{| v |^{p-2}}v~~~~{\text{ in }}{\mathbb{R}^3}  \\
v>0,~~~~v(0)=\max v,~~~~\int_{{\mathbb{R}^3}} {{v}^2}=c^2.
\end{gathered}  \right.
\end{equation}
Denote $I(v)\!:=\!\frac{1}{2}\|\nabla v\|_{2}^{2}\!-\!\frac{2|\lambda_{3}|}{p}\|v\|_{p}^{p}$, then solutions $v_c$ of \eqref{revisedI} can be found as minimizers of
$$ m_0(c)=\inf_{v \in S_c} I(v)>-\infty,$$
with $\beta_c$ appearing as Lagrange multipliers. Similar to the proof of Theorem 1.1 in \cite{NSoa}, we can prove that $m_0(c)$ is attained by some $v_c\in H^1(\mathbb{R}^3)$ by using the concentration-compactness principle. Applying the rearrangement technique (see \cite{ELMA}) and the strong maximum principle, we can assume that $v_c\in H_{rad}^1(\mathbb{R}^3)$, $v_c(0)=\max v_c$ and $v_c>0$. 

From \cite{Kwon,Wein}, we know that $W_p$ is the unique positive radial  solution of
$$ -\Delta W\!+\!(\frac{1}{\delta_p}-\!1)W \!=\!\frac{2}{p\delta_p}|W|^{p-2}W,~~~~\mbox{where}~~~~2\!<\!p\!<\!6. $$

\begin{lemma} \label{LeMa4.4.1}
Let $c\!>\!0$, $\lambda_{3}\!<\!0$ and $2\!<\!p\!<\!\frac{10}{3}$. Then \eqref{revisedI} has a unique positive solution $(\beta_c, v_c)$ given by
$$\beta_c:=(1-\delta_p)\Big[2\delta_p  \Big]^{\frac{p\delta_p}{2-p\delta_p}}
\Big[\mathcal{C}_{p}^{p}|\lambda_{3}|  \Big]^{\frac{2}{2-p\delta_p}}
c^{\frac{2(p-2)}{2-p\delta_p}},~~~~~~~~v_c(x):=\Big[  \frac{2\beta_c}{p(1-\delta_p)|\lambda_{3}|}\Big]^{\frac{1}{p-2}}W_p(\sqrt{\frac{2\delta_p\beta_c}{1-\delta_p}} x),$$
where $\delta_p\!=\!\frac{3(p-2)}{2p}$ and $\mathcal{C}_{p}\!=\!\Big( \frac{p}{2 ||W_p||^{p-2}_{2}} \Big)^{\frac{1}{p}}$. Furthermore, we have
$$m_0(c)=I(v_c)=-\kappa_{p,\lambda_{3}} c^{\frac{6-p}{2-p\delta_p}},$$
where $\kappa_{p,\lambda_{3}}\!:=\!\frac{10-3p}{6(p-2)}
\Big[ 2\delta_p \mathcal{C}_{p}^p|\lambda_{3}| \Big]
^{\frac{2}{2-p\delta_p}}\!>\!0$.
\end{lemma}
\begin{proof}
Since minimizers of $m_0(c)$ is positive and radially symmetric, and solving
$-\frac{1}{2} \Delta v+\beta_c v=(-{\lambda _3}){| v |^{p-2}}v$ for some $\beta_c>0$, it must be a scaling of $W_p$. The proof is similar to that of Proposition 2.2 in \cite{TblO} or Lemma 4.1 in \cite{Maed}, so we omit the details.
\end{proof}

\begin{lemma} \label{LeMA4.4.2}
Let $(\lambda_1, \lambda_2, \lambda_3)\!\in\! D_0\!\times\!\R^{-}$, $2\!<\!p\!<\!\frac{10}{3}$ and $0\!<\!c\!\leq\!c_*$. Then
$$m(c,R_0)<-\kappa_{p,\lambda_{3}} c^{\frac{6-p}{2-p\delta_p}}, $$
where $c_*$ is given in \eqref{cost2.11} and $\kappa_{p,\lambda_{3}}\!:=\!\frac{10-3p}{6(p-2)}
\Big[ 2\delta_p \mathcal{C}_{p}^p|\lambda_{3}| \Big]
^{\frac{2}{2-p\delta_p}}\!>\!0$.
\end{lemma}

\begin{proof}
We deduce from Lemma~\ref{2.2.} that
\begin{align} \label{iIqE2.3}
I(v)=\frac{1}{2}\|\nabla v\|_{2}^{2}-\frac{2|\lambda_{3}|}{p}\|v\|_{p}^{p} \geq g_c(\|\nabla v\|_{2}),~~~~\forall u\in S_c,
\end{align}
where $g_c(t):=\frac{1}{2}t^{2}-\frac{2|\lambda_{3}|\mathcal{C}_{p}^p c^{p(1-\delta_p)} }{p}t^{p\delta_p}$. Notice that $g_c(t)<0$ if $t\in(0,\bar{t}_c)$ and $g_c(t)>0$ if $t\in(\bar{t}_c,+\infty)$, where $\bar{t}_c:=\Big[\frac{4|\lambda_{3}|\mathcal{C}_{p}^p c^{p(1-\delta_p)} }{p}\Big]^{\frac{1}{2-p\delta_p}}$. If $I(v_c)=m_0(c)<0$, it follows immediately that $\|\nabla v_c\|_{2}<\bar{t}_c$.
Let $v_c$ be given by Lemma \ref{LeMa3.4.1}, then Lemma \ref{lem3.1} indicates that
$$
\|v_c\|_{2}=c,~~~~~~~~\|\nabla v_c\|_{2}=
\Big[ 2\delta_p \mathcal{C}_{p}^p|\lambda_{3}| \Big]
^{\frac{1}{2-p\delta_p}}c^{\frac{6-p}{10-3p}} <\bar{t}_c<R_0
$$
since $p\delta_p<2$. That is, $v_c\in V^c_{R_0}$ and hence we have
$$m(c,R_0)=\inf _{V^c_{R_0}} E \leq E(v_c)<I(v_c)=m_0(c)=-\kappa_{p,\lambda_{3}} c^{\frac{6-p}{2-p\delta_p}}. $$
\end{proof}

\section{Proof of Theorems \ref{th1.1}-\ref{th1.3}}
In this Section, we prove our main results.
For any $ k \in [R_{0}(c), R_{1}(c)]$, the problem $m(c,{t}_{c_*})=\inf _{V^c_{{t}_{c_*}}} E$ can be relaxed to
\begin{equation} \label{eAqn3.1}
m(c,k)=\inf _{u\in V^c_{k}} E(u),~~~~~~~~ \mbox{where}~~~~V^c_{k}:=\left\{ u \in S_c,||\nabla u||_{2}<k\right\}.
\end{equation}
(see Lemma \ref{lem3.1} (i)). Indeed, if $||\nabla u||_{2}\!\in\! [R_{0}, R_{1}]$, then $E(u) \!\geq\! h_c(||\nabla u||_{2})\!\geq\!0\!>\!\inf _{V^c_{{t}_{c_*}}} E$. 

\textbf{Proof of Theorem \ref{th1.1}} From Lemma \ref{LeMa3.4.1}, we see that
$m(c,{t}_{c_*})$ is attained by some $u_c \!\in\! V^c_{ {t}_{c_*}} \!:=\!\left\{u \!\in\! S_c :||\nabla u||_{2}\!<\! {t}_{c_*}\right\}$. Combined with \eqref{eAqn3.1}, we have $||\nabla u_c||_{2}\!<\!R_{0}$ and
$$ \inf _{u\in V^c_{R_{0}}} E(u)\!=\!m(c,R_{0})\!=\!m(c,  {t}_{c_*})\!=\!E(u_c). $$


Next, we prove that $u_c$ is a ground state and any ground state of $E|_{S_{c}}$ is a local minimizer of $E$ in $V^c_{R_0}$. By Lemma 2.2 in \cite{as}, we see that any critical point of
$E|_{S_{c}}$ lies in
 \begin{equation}  \label{eq2.4}
\mathcal{P}_{c}=\left\{v \in S_{c} : P(v):=2{||\nabla v||}_2^2+3B(v)+4\lambda_3\delta_{p}{||v||}_p^p=0\right\}.
\end{equation}
Therefore, $P(u_c)\!=\!0$ and $\mathcal{P}_{c}$ contains all the ground states of
$E_{\mu}|_{S_{c}}$. Let $v \!\in \!S_c$ be fixed, we have $v_s(x)\!:=\!s^{\frac{3}{2}} v\left(s x\right)\!\in\! S_c$ for any $s\!>\!0$. Consider the fiber maps
\begin{equation} \label{eq1.13}
\Psi_{v}(s) :=E(v_s)=\frac{s^2}{2} {||\nabla v||}_2^2+\frac{s^3}{2}B(v)-\frac{2|\lambda_{3}|s^{p\delta_{p}}}{p}{||v||}_p^p,
~~~~~~~~\forall s\!>\!0,
\end{equation}
we have
\begin{equation} \label{equa2.9}
\Psi_{v}^{\prime}(s)=\frac{d \Psi_{v}(s) }{ds}=s{||\nabla v||}_2^2+\frac{3s^2}{2}B(v)-2|\lambda_{3}|\delta_{p}s^{p\delta_{p}-1}{||v||}_p^p
=\frac{P(v_s)}{2s}.
\end{equation}
Consequently, $\Psi_{v}^{\prime}(s)=0$ reads
$$\varphi_c(s)=2|\lambda_{3}|\delta_{p}{||v||}_p^p,~~~~\mbox{with}~~~~ \varphi_c(s)=s^{2-p\delta_{p}}{||\nabla v||}_2^2+\frac{3s^{3-p\delta_{p}}}{2}B(v). $$
We observe that $\varphi_c\nearrow$ on $(0,\bar{s})$, $\searrow$ on $(\bar{s},+\infty)$, where $\bar{s}=\frac{2(2-p\delta_{p}){||\nabla v||}_2^2}{3(3-p\delta_{p})[-B(v)]}$. As $c\!\leq\!c_*\!<\!c^*$ (see \eqref{c^*}), the Gagliardo-Nirenberg inequality \eqref{equ2.2} leads to
\begin{align*}
\varphi_c(\bar{s})&=\frac{\bar{s}^{2-p\delta_{p}} }{3-p\delta_{p}}{||\nabla v||}_2^2 \geq \frac{1}{(3-p\delta_{p})}\Big[    \frac{2(2-p\delta_{p})}{3(3-p\delta_{p})\Lambda \mathcal{C}_{4}^4 } \Big]^{2-p\delta_{p}} \frac{{||\nabla v||}_2^{p\delta_{p}}}{c^{2-p\delta_{p}}} \\
&>2|\lambda_{3}|\delta_{p}\mathcal{C}_{p}^p \left\|\nabla u\right\|_{2}^{p\delta_p} c^{p(1-\delta_p)}
\geq 2|\lambda_{3}|\delta_{p}{||v||}_p^p.
\end{align*}  
Since $\varphi_c(0^{+})\!=\!0^{+}$, $\varphi_c(+\infty)\!=\!-\infty$, we see that $\Psi_{v}$ has exactly two critical points. From (\ref{eq2.3}), we have $\Psi_{v}(s)\!=\!E(v_s) \!\geq\! h_c\left(||\nabla v_s||_{2}\right)\!=\!h_c\left(s||\nabla v||_{2}\right)$. So, we get
$$\Psi_{v}(s)\!=\!E(v_s) \!\geq\! h_c\left(s||\nabla v||_{2}\right)\!>\!0,~~~~~~\forall s\!\in\!\big( \frac{R_{0}}{||\nabla v||_{2}},  \frac{R_{1}}{||\nabla v||_{2}}\big)$$
if $c\!<\!c_*$ and $\Psi_{v}(\frac{ {t}_{c_*}}{||\nabla v||_{2}})\!\geq\! h_{c_*}({t}_{c_*})\!=\!0$ if $c\!=\!c_*$, and clearly $\Psi_{v}(0^{+})=0^{-}$, $\Psi_{v}(+\infty)=-\infty$. It follows that $\Psi_{v}$ has a local minimum point $s_v$ on $(0, \frac{R_{0}}{||\nabla v||_{2}})$ at negative level and a local maximum point $t_v$  at nonnegative level such that $s_v\!<\!t_v$. 

Let $w$ be a ground state of $E|_{S_{c}}$, then $P(w)\!=\!0$ and $E(w)\!\leq\! E(u_c)\!<\!0$. Furthermore, there exist a $v\!\in\!S_c$ and a $s_w\!=\!\frac{||\nabla w||_{2}}{||\nabla v||_{2}} \!\in\!(0,+\infty)$ such that $w(x)\!=\!v_{s_w}(x)\!=\!s_w^{\frac{3}{2}} v({s_w} x)\!\in \!S_c$. By using \eqref{eq1.13}-\eqref{equa2.9} with the previous selected $v\!\in\!S_c$, we have
$$\Psi_{v}(s_w)=E(v_{s_w})=E(w)<0,~~~~~~~~\Psi_{v}'(s_w)=P(v_{s_w})
=P(w)=0.$$
It must be that $s_w\!=\!s_{v}\!<\!\frac{R_{0}}{||\nabla v||_{2}}$. Hence, we have $||\nabla w||_{2}\!<\!R_{0}$ and $w\!=\!v_{s_{v}}\!\in\! V^c_{R_0}$. Also, we see that $u_c$ is a ground state of $E|_{S_{c}}$ since $E(w)\!\leq\! E(u_c)\!=\!\inf _{u\in V^c_{R_{0}}} E(u)\!\leq\!E(w)$.


Next, the Lagrange multipliers rule implies the existence of some $\mu_c \in \mathbb{R}$ such that
\begin{align} \label{eq4.2.1}
\frac{1}{2}\int_{\mathbb{R}^{3}}  \nabla u_c \cdot \nabla {\varphi}&+{\lambda _1}  \int_{\mathbb{R}^{3}}{|u_c|^2}u_c {\varphi}+{\lambda _2} \int_{\mathbb{R}^{3}}(K\star{|u_c|^2})u_c{\varphi} \nonumber \\
&+\lambda_3\int_{\mathbb{R}^{3}} \left|u_c\right|^{p-2} u_c {\varphi}+\mu_c\int_{\mathbb{R}^{3}}u_c {\varphi}=0
\end{align}
for each $\varphi \in H^{1}(\mathbb{R}^{3})$. That is, $u_c$ satisfies
$$
- \frac{1}{2} \Delta u +{\lambda _1}{| u |^2}u + {\lambda _2}(K \star {| u |^2})u+{\lambda _3}{| u |^{p-2}}u+\mu_c u=0.
$$
By using Lemma \ref{LeMA4.4.2}, we can prove that
$$   \kappa_{p,\lambda_{3}}c^{\frac{2(p-2)}{2-p\delta_p}}\!<\!\mu_{c}  \!<\!\frac{1-\delta_p}{2\delta_p}
\Big[\frac{4(3-p\delta_p)|\lambda_{3}|\mathcal{C}_{p}^p }{p}\Big]^{\frac{2}{2-p\delta_p}} c^{\frac{2(p-2)}{2-p\delta_p}}.$$
In fact, taking $\varphi\!=\!u_c$ in \eqref{eq4.2.1}, we have $\mu_{c}\!>\!\kappa_{p,\lambda_{3}}c^{\frac{2(p-2)}{2-p\delta_p}}$ as
\begin{equation*}
\mu_{c}c^{2}\!=\!-\Big[\frac{1}{2}{||\nabla u_c||}_2^2\!+\!B(u_c)\!+\!\lambda_3{||u_c||}_p^p\Big]\!=\!-E(u_c)\!-\!\frac{1}{2}B(u_c)
\!+\!\frac{(p-2)|\lambda_3|}{p}{||u_c||}_p^p\!>\!\kappa_{p,\lambda_{3}} c^{\frac{6-p}{2-p\delta_p}}.
\end{equation*}
Since $P(u_c)\!=\!0$, we get
\begin{equation*}
E\left(u_{c}\right)\!=\!\frac{1}{6}||\nabla u_{c}||_{2}^{2}-\frac{2(3-p\delta_{p})|\lambda_{3}|} {3p}  ||u_{c}||_{p}^{p}\!=\!\left(\frac{1}{2}\!-\!\frac{1}{p\delta_{p} }\right)  ||\nabla u_{c}||_{2}^{2}+\frac{3-p\delta_{p}}{2p\delta_{p} } |B(u_c)|\!<\!-\kappa_{p,\lambda_{3}}c^{\frac{6-p}{2-p\delta_p}}.
\end{equation*}
It follows immediately that
\begin{equation} \label{gradientboud}
\frac{2p\delta_p}{2-p\delta_p} \kappa_{p,\lambda_{3}} c^{\frac{6-p}{2-p\delta_p}}<||\nabla u_{c}||_{2}^2<\Big[\frac{4(3-p\delta_p)|\lambda_{3}|\mathcal{C}_{p}^p }{p}\Big]^{\frac{2}{2-p\delta_p}} c^{\frac{6-p}{2-p\delta_p}}.
\end{equation}
Hence we obtain
\begin{align*}
\mu_c c^2&=\frac{1-\delta_p}{2\delta_p}||\nabla u_{c}||_{2}^{2}\!-\!\frac{4-p}{2(p-2)}|B(u_c)|<\frac{1-\delta_p}{2\delta_p}
\Big[\frac{4(3-p\delta_p)|\lambda_{3}|\mathcal{C}_{p}^p }{p}\Big]^{\frac{2}{2-p\delta_p}} c^{\frac{6-p}{2-p\delta_p}}.
\end{align*}
Since $|| \nabla  |u_c| ||_2  \leq  \|\nabla u_c\|_2$ (see \cite{ELMA}), we can assume that $u_c\geq0$ and the strong maximum principle implies that $u_c>0$. Thus, Theorem \ref{th1.1} \textbf{(1)} \textbf{(2)} are true. From Lemma \ref{lem3.1} , we know that $R_{0}(c) \rightarrow 0$ as $\lambda_3 \rightarrow 0^{-}$, and hence $||\nabla u_{c}||_{2}<R_{0}(c) \rightarrow 0$ as well. Moreover
\begin{align*}
0>m(c,R_{0}) \!\geq\! \frac{1}{2}\|\nabla u_{c}\|_{2}^{2}-\frac{\Lambda \mathcal{C}_{4}^4 c }{2} \left\|\nabla u_{c}\right\|_{2}^{3} -\frac{2|\lambda_{3}|\mathcal{C}_{p}^p c^{p(1-\delta_p)} }{p}  \left\|\nabla u_{c}\right\|_{2}^{p\delta_p} \rightarrow 0,
\end{align*}
which implies that $m(c,R_{0})\to 0^-$ as $\lambda_3 \rightarrow 0^{-}$. Hence, \textbf{(3)} of Theorem \ref{th1.1} follows.


We now prove the stability of $Z_{c}$ under the flow corresponding to problem \eqref{gpt1.4}. Suppose that there exists $\varepsilon>0$, a sequence of initial data $\{\psi_{n,0} \} \subset H^1(\R^3)$ and a sequence $\{t_n\}\subset(0,T_{\psi_n}^{max})$ such that the maximal solution $\psi_n(t,x)$ with $\psi_n(0,x)=\psi_{n,0}(x)$ satisfies
\begin{align} \label{al5.9}
\lim _{n \rightarrow \infty} \inf_{v \in Z_{c}} ||\psi_{n,0}-v||_{H^1(\R^3)}=0,~~~~\mbox{and}~~~~ \inf_{v \in Z_{c}} ||\psi_{n}(t_n,\cdot)-v||_{H^1(\R^3)}\geq \varepsilon.
\end{align}
Since $c\!\mapsto \! m(c,  {t}_{c_*})$ is continuous, we have 
$$||\psi_{n,0}||_2=:c_n\to c,~~~~~~~~E(\psi_{n,0}) \to m(c,  {t}_{c_*})=m(c,R_0)\!<\!-\kappa_{p,\lambda_{3}}c^{\frac{6-p}{2-p\delta_p}}.$$
For each $n$ sufficiently large, we have $E(\psi_{n,0})<0$. Therefore,
we deduce from Lemma \ref{lem3.1} (i) (ii) that $||\nabla \psi_{n,0}||_2<R_0(c_n)\leq {t}_{c_*}$.

Let us consider now the solution $\psi_n(t,\cdot)\in S_{c_n}$. If there exists $t_0\in(0, T_{\psi_n}^{max})$ such that $||\nabla \psi_n(t_0,\cdot)||_2= {t}_{c_*}$, then $E(\psi_n(t_0, \cdot))\geq
h_{c_n}( {t}_{c_*})\geq0$, which is impossible since $m(c, {t}_{c_*})\!<\!-\kappa_{p,\lambda_{3}}c^{\frac{6-p}{2-p\delta_p}}$. This shows that $||\nabla \psi_{n}(t,\cdot)||_2<{t}_{c_*}$ for every $t\in[0,T_{\psi_n}^{max})$. Moreover, by conservation of mass and of energy
$$||\psi_n(t_n, \cdot)||_2\to c,~~~~~~~~E(\psi_{n}(t_n, \cdot)) \to m(c, {t}_{c_*}).$$
It follows that $\{\psi_n(t_n,\cdot)\}$ is relatively compact up to translations in $H^{1}(\R^3)$, and hence it converges, up to a translation, to a ground state in $Z_{c}$, in contradiction with (\ref{al5.9}). \qed


\vspace{0.5cm}

\noindent \textbf{Proof of Theorem \ref{th1.3}:}

The first part of our proof was motivated by Theorem 1.3 in \cite{Nbal} which studied a Biharmonic equation. Let $c_k \!\to\! 0^+$ as $k\!\to\!+\infty$ and $u_{c_k}\!\in\! V^{c_k}_{R_0}$ be a positive minimizer of $m(c_k,R_0)$ for each $k\!\in\! \mathbb{N}$, where $V^{c_k}_{R_0}\!:=\!\left\{u \in S_{c_k} : {||\nabla u||}_2\!<\!R_0(c_k) \right\}$. 
Similar to \eqref{gradientboud}, we get
$$\frac{2p\delta_p\kappa_{p,\lambda_{3}}}{2-p\delta_p} <\frac{\left\|\nabla u_{c_k}\right\|_{2}^2}{c_k^{\frac{6-p}{2-p\delta_p}}}
<\Big[\frac{4(3-p\delta_p)|\lambda_{3}|\mathcal{C}_{p}^p }{p}\Big]^{\frac{2}{2-p\delta_p}}.$$
Then, Lemma \ref{2.2.} and the Gagliardo-Nirenberg inequality \eqref{equ2.2} lead  to
\begin{align} \label{aLn5.7}
\frac{ |B(u_{c_k})| }{c_k^{\frac{6-p}{2-p\delta_p}}} &\leq \frac{\Lambda \left\| u_{c_k} \right\|_{4}^{4}}{c_k^{\frac{6-p}{2-p\delta_p}}}\leq
\Lambda \mathcal{C}_{4}^4\Big[ \frac{\left\|\nabla u_{c_k} \right\|_{2}^2}{c_k^{\frac{6-p}{2-p\delta_p}}} \Big]^{\frac{3}{2}}c_k^{\frac{4(4-p)}{10-3p}} \nonumber \\
&\leq\Lambda\mathcal{C}_{4}^4\Big[\frac{4(3-p\delta_p)
|\lambda_{3}|\mathcal{C}_{p}^p }{p}\Big]^{\frac{3}{2-p\delta_p}}c_k^{\frac{4(4-p)}{10-3p}}\to0
\end{align}
as $k\!\to\!+\infty$. We also have $m(c_k,R_0)\!<\!-\kappa_{p,\lambda_{3}} c_k^{\frac{6-p}{2-p\delta_p}}$ by Lemma \ref{LeMA4.4.2}. These facts imply that
\begin{align} \label{aLn5.8}
-\kappa_{p,\lambda_{3}} c_k^{\frac{6-p}{2-p\delta_p}}\!&>\!m(c_k,R_0)=\!E(u_{c_k})\!=\!\frac{1}{2}\|\nabla u_{c_k} \|_{2}^{2}+\frac{1}{2}B(u_{c_k})-\frac{2|\lambda_{3}|}{p}\|u_{c_k}\|_{p}^{p} \nonumber \\
&\geq \! \inf _{u \in S_{c_k}} \Big\{ \frac{1}{2}\|\nabla v \|_{2}^{2}-\frac{2|\lambda_{3}|}{p}\|v\|_{p}^{p} \Big\}\!+\!\frac{1}{2}B(u_{c_k}) \!=\! -\kappa_{p,\lambda_{3}} c_k^{\frac{6-p}{2-p\delta_p}}\!+\!\frac{1}{2}B(u_{c_k}),
\end{align}
where we use Lemma \ref{LeMa4.4.1} in the last equality. From \eqref{aLn5.7}-\eqref{aLn5.8}, we see that
\begin{align} \label{aLn5.9}
\frac{m(c_k,R_0)}{c_k^{\frac{6-p}{2-p\delta_p}}}  \to  -\kappa_{p,\lambda_{3}},~~~~~~~~~~~~~~~~~~ \frac{1}{2}\frac{\left\|\nabla u_{c_k} \right\|_{2}^2}{c_k^{\frac{6-p}{2-p\delta_p}}} -\frac{2|\lambda_{3}|}{p}\frac{\left\| u_{c_k} \right\|_{p}^p}{c_k^{\frac{6-p}{2-p\delta_p}}}\to-\kappa_{p,\lambda_{3}}.
\end{align}
As $P(u_{c_k})=0$, we get
$0\!=\!\frac{P(u_{c_k})}{c_k^{\frac{6-p}{2-p\delta_p}}}\!=\!2\frac{{||\nabla u_{c_k} ||}_2^2}{c_k^{\frac{6-p}{2-p\delta_p}}}\!-\!3\frac{B(u_{c_k})}
{c_k^{\frac{6-p}{2-p\delta_p}}}\!-\!4\delta_{p}|\lambda_3|
\frac{{||u_{c_k}||}_p^p}{c_k^{\frac{6-p}{2-p\delta_p}}}$ and hence we get
\begin{align} \label{aln5.10}
 2\frac{{||\nabla u_{c_k} ||}_2^2}{c_k^{\frac{6-p}{2-p\delta_p}}}\!-\!4\delta_{p}|\lambda_3|
\frac{{||u_{c_k}||}_p^p}{c_k^{\frac{6-p}{2-p\delta_p}}}\to0.
\end{align}
Consequently, \eqref{aLn5.9}-\eqref{aln5.10} indicate that
$$
\frac{\left\|\nabla u_{c_k} \right\|_{2}^2}{c_k^{\frac{6-p}{2-p\delta_p}}} \to
\frac{2p\delta_p}{2-p\delta_p} \kappa_{p,\lambda_{3}},~~~~~~~~
\frac{\left\|u_{c_k}\right\|_{p}^p}{c_k^{\frac{6-p}{2-p\delta_p}}}\to \frac{p}{(2-p\delta_p)|\lambda_{3}|}\kappa_{p,\lambda_{3}}.
$$
Furthermore, we have
$$\frac{\mu_{c_k}}{c_k^{\frac{2(p-2)}{2-p\delta_p}}}=\frac{-1}{c_k^{\frac{6-p}{2-p\delta_p}}}\Big[\frac{1}{2}{||\nabla u_{c_k}||}_2^2+B(u_{c_k})+\lambda_3{||u_{c_k}||}_p^p\Big]\to \frac{p(1-\delta_p)}{2-p\delta_p}\kappa_{p,\lambda_{3}}.$$

Next, we give a precise description of $u_{{c_k}}$ as $k\!\to\!+\infty$. Denote
$$ a_k:=\Big[  \frac{2\beta_{c_k}}{p(1-\delta_p)|\lambda_{3}|}\Big]^{\frac{1}{p-2}},
~~~~~~~~b_k:=\sqrt{\frac{2\delta_p\beta_{c_k}}{1-\delta_p}},~~~~~~\beta_{c_k}:=(1-\delta_p)\Big[2\delta_p  \Big]^{\frac{p\delta_p}{2-p\delta_p}}
\Big[\mathcal{C}_{p}^{p}|\lambda_{3}|  \Big]^{\frac{2}{2-p\delta_p}}
c_k^{\frac{2(p-2)}{2-p\delta_p}}$$
and define $v_k(x)\!:=\!a_k^{-1}u_{{c_k}}(b_k^{-1}x)$, we then compute that
\begin{align} \label{NewbOUDS1}
\left\|\nabla v_{k}\right\|_{2}^2=\frac{b_k}{a_k^2}\left\|\nabla u_{c_k}\right\|_{2}^2=p^{\frac{2}{p-2}}\Big[2\mathcal{C}_{p}^{p} \Big]^{\frac{-(6-p)}{(2-p\delta_p)(p-2)}}   \Big[\delta_p|\lambda_{3}|  \Big]^{\frac{-2}{2-p\delta_p}} \frac{\left\|\nabla u_{c_k}\right\|_{2}^2}{c_k^{\frac{6-p}{2-p\delta_p}}},
\end{align}
\begin{align} \label{NewbOUDS2}
\left\| v_{k}\right\|_{2}^2=\frac{b_k^3}{a_k^2}\left\| u_{c_k}\right\|_{2}^2=\Big[\frac{p}{2\mathcal{C}_{p}^{p}} \Big]^{\frac{2}{p-2}},
\end{align}
\begin{align} \label{NewbOUDS3}
\left\| v_{k}\right\|_{p}^p=\frac{b_k^3}{a_k^p}\left\| u_{c_k}\right\|_{p}^p=\Big[\frac{p}{2}\Big]^{\frac{p}{p-2}}
\Big[\mathcal{C}_{p}^{p} \Big]^{\frac{-(6-p)}{(2-p\delta_p)(p-2)}}   \Big[2\delta_p|\lambda_{3}|  \Big]^{\frac{-p\delta_p}{2-p\delta_p}} \frac{\left\|u_{c_k}\right\|_{p}^p}{c_k^{\frac{6-p}{2-p\delta_p}}} .
\end{align}
As a result of \eqref{NewbOUDS1}-\eqref{NewbOUDS2}, $\{v_k\}$ is bounded in $H^1(\R^3)$.
If $\mathop {\lim }\limits_{k  \to \infty} \sup _{y \in \mathbb{R}^{3}} \int_{B_{R}(y)}\left|v_{k}(x)\right|^{2}dx\!=\!0$ for any $R\!>\!0$, we can prove that $\left\|v_{k}\right\|_{r} \to 0$ since $2<r<6$, which contradicts with \eqref{NewbOUDS3}. Then, there exist an $\varepsilon_0>0$ and a sequence $\{y_{k}\} \subset \mathbb{R}^{3}$ such that
$$
\int_{B_{R}(y_k)}\left|v_{k}(x)\right|^{2} d x \geq \varepsilon_0>0
$$
for some $R>0$. Hence, up to a subsequence, we have $$\tilde{v}_{k}(x):=v_{k}(x+y_k)\rightharpoonup v \not\equiv0~~~~~~~~\mbox{in}~~~~~~~~H^1(\R^3)$$
for some $v \in H^1(\R^3)$ with $v\geq0$. Since $u_{c_k}$ satisfies
$$
- \frac{1}{2} \Delta u_{c_k} +{\lambda _1}{| u_{c_k} |^2}u_{c_k} + {\lambda _2}(K \star {| u_{c_k} |^2})u_{c_k}+{\lambda _3}{| u_{c_k} |^{p-2}}u_{c_k}+\mu_{c_k} u_{c_k}=0,
$$
we see that $\tilde{v}_{k}$ solves
\begin{align} \label{equaV}
- \frac{1}{2} \Delta \tilde{v}_{k} +\frac{a_k^2}{b_k^2}\Big({\lambda _1}{| \tilde{v}_{k} |^2}\tilde{v}_{k} + {\lambda _2}(K \star {| \tilde{v}_{k} |^2})\tilde{v}_{k}\Big)+{\lambda _3}\frac{a_k^{p-2}}{b_k^2}{| \tilde{v}_{k} |^{p-2}}\tilde{v}_{k}+\frac{\mu_{c_k}}{b_k^2} \tilde{v}_{k}=0.
\end{align}
By the previous asymptotic properties, we also deduce that
\begin{align} \label{NewbOUDS4}
&\frac{a_k^2}{b_k^2}=\Big[\frac{2}{p}\Big]^{\frac{2}{p-2}}
\Big[\mathcal{C}_{p}^{p} \Big]^{\frac{-4}{2-p\delta_p}}   \Big[2\delta_p|\lambda_{3}|\Big]^{\frac{1}{2-p\delta_p}} c_k^{\frac{2(4-p)}{2-p\delta_p}} \to 0, \nonumber \\
& \frac{a_k^{p-2}}{b_k^2}=\frac{1}{p\delta_p|\lambda_{3}|},
~~~~~~~~\frac{\mu_{c_k}}{b_k^2}=\Big[2\delta_p|\lambda_{3}|\mathcal{C}_{p}^{p}  \Big]^{\frac{-2}{2-p\delta_p}} \frac{\mu_{c_k}}{c_k^{\frac{2(p-2)}{2-p\delta_p}}} \to \frac{1-\delta_p}{2\delta_p}.
\end{align}
Therefore, $v$ solves
\begin{align} \label{Equav}
- \frac{1}{2} \Delta v-\frac{1}{p\delta_p}{|v|^{p-2}}v+\frac{1-\delta_p}{2\delta_p}v=0.
\end{align}
Since $v\geq0$ and $v \not\equiv0$, the strong maximum principle implies that $v>0$. Similar to the proof of Theorem 2.4 in \cite{RpoL}, we have $v(x)\to0$ and $|x| \to+\infty$. Therefore, up to a translation, $v$ is radially symmetric about $0\in\R^3$ (see \cite{LiMi}). From \cite{Kwon,Wein}, we know that the equation
$$ -\Delta W\!+\!(\frac{1}{\delta_p}-\!1)W \!=\!\frac{2}{p\delta_p}|W|^{p-2}W, $$
has a unique positive radial solution $W_p$, so it must be $v=W_p$ (up to a translation). Observing that
$$
\left\| {v}\right\|_{2}^2 \leq \mathop {\lim }\limits_{k  \to \infty}\left\| \tilde{v}_{k}\right\|_{2}^2=\mathop {\lim }\limits_{k  \to \infty}\left\| {v}_{k}\right\|_{2}^2=\Big[\frac{p}{2\mathcal{C}_{p}^{p}} \Big]^{\frac{2}{p-2}}=\left\|W_p\right\|_{2}^2=\left\| {v}\right\|_{2}^2,
$$
we have $\tilde{v}_{k}(x)\rightarrow v$ in $L^2(\R^3)$. By the Gagliardo-Nirenberg inequality \eqref{equ2.2}, $\tilde{v}_{k}(x)\rightarrow W_p$ in $L^r(\R^3)$ for $r\in(2,6)$. As a result of \eqref{equaV} and \eqref{Equav}, we have $\tilde{v}_{k}(x)\rightarrow W_p $ in $H^1(\R^3)$. Moreover, as the limit function $W_p$ is independent of the subsequence that we choose, we see that the convergence is true for the whole sequence.  For simplicity, denote  $\gamma_{c_k}\!:=\!\frac{\beta_{c_k}}{1-\delta_p}$ and then we have $\tilde{v}_k(x)\!:=\!\Big[  \frac{p|\lambda_{3}|}{2\gamma_{c_k}}
\Big]^{\frac{1}{p-2}}u_{{c_k}}(\frac{x+y_k}{\sqrt{2\delta_p\gamma_{c_k}}})\!
\rightarrow\! W_p $ in $H^1(\R^3)$.  \qed \\

\vspace{.55cm}


\begin{thebibliography}{99999}


\bibitem{AlGi}
A. Alvino, G. Trombetti, J. I. Diaz, P. L. Lions, Elliptic equations and Steiner Symmetrization, Commun. Partial. Diff. Equ. 49 (1996) 217-236.

\bibitem{as}
P. Antonelli, C. Sparber, Existence of solitary waves in dipolar quantum gases, Physica D 240 (2011) 426-431.


\bibitem{a7}
H. Br\'{e}zis, E. Lieb, A relation between pointwise convergence of functions and convergence of functionals, Proc. Amer. Math. Soc 88 (1983) 486-490.

\bibitem{Nbal}
N. Boussaid, A. J. Fernandez, L. Jeanjean, Some remarks on a minimization problem associated to a fourth order nonlinear Scrh\"{o}dinger equation, Preprint, arXiv:1910.13177.


\bibitem{bcd}
H. Bahouri, J-Y. Chemin, R. Danchin, Fourier Analysis and Nonlinear Partial Differential Equations, Springer, Heidelberg, 2011.

\bibitem{TblO}
T. Bartsch, L. Jeanjean, N. Soave, Normalized solutions for a system of
coupled cubic Schr\"odinger equations on $\R^3$, J. Math. Pures Appl. 106 (2016) 583-614.

\bibitem{lAkI}
P. B. Blakie, Properties of a dipolar condensate with three-body interactions, Phys. Rev. A 93 (2016) 033644.

\bibitem{bcw}
W. Bao, Y. Cai, H. Wang, Efficient numerical method for computing ground states and dynamic of dipolar Bose-Einstein condensates, J. Comput. Phys. 229 (2010) 7874-7892.

%





\bibitem{bj1}
J. Bellazzini, L. Jeanjean, On dipolar quantum gases in the unstable regime, SIAM J. Math. Anal. 48 (2017) 2028-2058.

\bibitem{ch}
R. Carles, H. Hajaiej, Complementary study of the standing waves solutions of the Gross-Pitaevskii equation in dipolar quantum gases, Bull. London Math. Soc. 47 (2015) 509-518.

\bibitem{cms}
R. Carles, P. Markowich, C. Sparber, On The Gross-Pitaevskii equation for trapped dipolar quantum gases, Nonlinearity 21 (2008) 2569-2590.


\bibitem{RpoL}
R. Cipolatti, On the existence of standing waves for a Davey-Stewartson system. Commun. Partial Diff. Equ. 17 (1992) 967-988.


\bibitem{GyWy}
Y. J. Guo, Y. Luo, W. Yang, The Nonexistence of Vortices for Rotating Bose-Einstein Condensates with Attractive Interactions, Arch. Rational Mech. Anal. 238 (2020) 1231-1281.

\bibitem{GLwJ}
Y. J. Guo, C. S. Lin, J. C. Wei, Local uniqueness and refined spike profiles of ground states for two-dimensional attractive Bose-Einstein condensates, SIAM J. Math. Anal. 49 (2017) 3671-3715.


\bibitem{YhXl}
Y. He, X. Luo, Concentrating standing waves for the Gross-Pitaevskii equation in trapped dipolar quantum gases, J. Differential Equations, 266 (2019) 600-629.

\bibitem{eHtR}
L. Jeanjean, J. Jendrej, T. T. Le, N. Visciglia, Orbital stability of ground states for a Sobolev critical schr\"odinger equation, arXiv:2008.12084v1.

\bibitem{Kwon}
M. K. Kwong, Uniqueness of positive solutions of $\Delta u -u + u^p=0$ in $\mathbb{R}^n$,  Arch. Rational Mech. Anal. 105 (1989) 243-266.

\bibitem{ELMA}
E. H. Lieb, M. Loss, Analysis, 2nd ed., Graduate Series in Mathematics 14, Amer. Math. Soc. Providence (2001).

\bibitem{LiMi}
Y. Li, W.-M. Ni, Radial symmetry of positive solutions of nonlinear
elliptic equations in $\R^n$,  Commun. Partial Diff. Equ. 18 (1993) 1043-1054.



\bibitem{LilP}
P. L. Lions, The concentration-compactness principle in the calculus of variations. The locally compact case. II. Ann. Inst. H.
Poincar\'{e} Anal. Non Lin\'{e}aire. 1 (1984) 223-283.


\bibitem{Ymas}
Y. M. Luo, A. Stylianou, Ground states for a nonlocal cubic-quartic
Gross-Pitaevskii equation, arXiv:1806.00697v3.

\bibitem{lYmAs}
Y. M. Luo, A. Stylianou, On 3d dipolar Bose-Einstein condensates involving quantum fluctuations and three-body interactions, Discrete Continuous Dynamical Systems-B, doi: 10.3934/dcdsb.2020239.

\bibitem{mBlo}
B. A. Malomed, Suppression of quantum-mechanical collapse in bosonic gases with intrinsic repulsion: A brief review, Condensed Matter 3 (2018) 15.


\bibitem{Maed}
M. Maeda, On the symmetry of the ground states of nonlinear Schr\"odinger
equation with potential, Adv. Nonlinear Stud. 10 (2010) 895-925.

\bibitem{MeLf}
J. Metz, T. Lahaye, B. Fr\"ohlich, A. Griesmaier, T. Pfau, H. Saito, Y. Kawaguchi, and M. Ueda. Coherent collapses of dipolar bose-einstein condensates for different trap geometries, New Journal of Physics, 11 (2009) 055032.









\bibitem{rSaU}
W. Strauss, Existence of Solitary Waves in Higher Dimensions, Comm. Math. Phys. 55 (1977) 149-162.

\bibitem{sszl}
L. Santos, G. Shlyapnikov, P. Zoller, M. Lewenstein, Bose-Einstein condensation in trapped dipolar gases, Phys. Rev. Lett. 85 (2000) 1791-1797.

\bibitem{NSoa}
N. Soave, Normalized ground states for the NLS equation with combined nonlinearities, J. Differential Equations 269 (2020) 6941-6987.

\bibitem{TrIa}
A. Triay, Existence of minimizers in generalized Gross-Pitaevskii
theory with the Lee-Huang-Yang correction, arXiv: 1904.10672v1.




\bibitem{Wein}
M. Weinstein, Nonlinear Schr\"odinger equations and sharp interpolation estimates, Commun. Math. Phys. 87 (1983) 567-576.


\bibitem{yy1}
S. Yi, L. You, Trapped atomic condensates with anisotropic interactions, Phys. Rev. 61 (2000) 041604.

\bibitem{yy2}
S. Yi, L. You, Trapped condensates of atoms with dipole interactions, Phys. Rev. 63 (2001) 053607.











%



%
%
%
%




























%
%
%



%


%
%
%
%
%
%
%
%
%





%






\end{thebibliography}
\end{document}